\documentclass[12pt,twoside]{amsart}
\usepackage{amssymb,amscd}
\usepackage{mathrsfs}
\usepackage{array,float}
\usepackage[all]{xy}
\usepackage{enumerate}

\theoremstyle{plain}
\newtheorem{thm}{Theorem}[section]
\newtheorem{lem}[thm]{Lemma}
\newtheorem{pro}[thm]{Proposition}
\newtheorem{cor}[thm]{Corollary}

\theoremstyle{remark}
\newtheorem{rem}[thm]{Remark}

\theoremstyle{definition} 
\newtheorem{definition}[thm]{Definition}

\numberwithin{equation}{section}

\newcommand{\Z}{\mathbb{Z}}   

\newcommand{\N}{\mathbb{N}}

\newcommand{\ackn}{  \noindent{\sc Acknowledgement }\hspace{5pt} }

\renewcommand{\phi}{\varphi}

\begin{document}

\author{Ilir Snopce}
\address{Universidade Federal do Rio de Janeiro\\
  Instituto de Matem\'atica \\
  21941-909 Rio de Janeiro, RJ \\ Brasil }
\email{ilir@im.ufrj.br}

\thanks{This research was partially supported by CNPq}

\title[Test elements in pro-$p$ groups with applications in discrete groups]{Test elements in pro-$p$ groups with applications in discrete groups}

\author{Slobodan Tanushevski} \address{Universidade Federal do Rio de Janeiro\\
  Instituto de Matem\'atica \\
  21941-909 Rio de Janeiro, RJ \\ Brasil }

\email{tanusevski1@gmail.com}

\begin{abstract}
Let $G$ be a group. An element $g \in G$ is called a test element of $G$ if for every endomorphism $\varphi:G \to G$, $\varphi(g)=g$ implies that $\varphi$ is an automorphism.
We prove that for a finitely generated profinite group $G$, $g \in G$ is a test element of $G$ if and only if it is not contained in a proper retract of $G$.
Using this result we prove that an endomorphism of a free pro-$p$ group of finite rank which preserves an automorphic orbit of a non-trivial element must be an automorphism.
We give numerous explicit examples of test elements in free pro-$p$ groups and Demushkin groups. By relating test elements in finitely generated residually finite-$p$ Turner groups to test elements in their pro-$p$
completions, we provide new examples of test elements in free discrete groups and surface groups. Moreover, we prove that the set of test elements of a free discrete group of finite rank is dense in the profinite topology.
\end{abstract}


\maketitle

\section{Introduction}
\label{introduction}
%Throughout let $p$ be a prime.  
%Let $G$ be a group. An element $g\in G$ is called a \emph{test element} of $G$ if for any endomorphism  $\varphi: G \to G$ ,  $\varphi(g) = g$   
%implies that  $\varphi$  is an automorphism.
A \emph{test element} of a group $G$ is an element $g \in G$ with the following property: if $\varphi(g)=g$ for an endomorphism  $\varphi: G \to G$ ,  then $\varphi$ must be an automorphism. 
Alternatively, $g \in G$ is a test element of $G$ if for every $\phi \in End(G)$,  $\varphi (g) = \alpha(g)$ for some $\alpha \in \textrm{Aut}(G)$ implies $\varphi \in \textrm{Aut}(G)$.  
The notion of a test element was introduced by Shpilrain \cite{Shpilrain} as a way of organizing into a general framework a plethora of hitherto scattered results.

Let $D(x_1, \ldots, x_n)$ be a free discrete group with basis $\{x_1, \ldots, x_n \}$. In 1918, Nielsen \cite{Nielsen} proved that the commutator $[x_1, x_2]$ is a test element of the free group $D(x_1, x_2)$ of rank two. This result was generalized by Zieschang \cite{Zieschang1}, who showed that $[x_1, x_2][x_3,x_4]\cdots [x_{2m-1}, x_{2m}]$ is a test element of $D( x_1, x_2, ..., x_{2m})$. 
Rips \cite{Rips} achieved  another generalization by proving that every higher commutator $[x_1, x_2, ... , x_n]$ is a test element of $D(x_1, ..., x_n)$. 
In another direction, $x_1^{k_1}x_2^{k_2}\cdots x_n^{k_n}$  is a test element of $D(x_1, ..., x_n)$ if and only if $k_i \neq 0$ for all $1 \leq i \leq n$ and $\gcd(k_1, \ldots, k_n) \neq 1$.
This was proved by Zieschang \cite{Zieschang2} for  the special case $k_1=k_2= \ldots = k_n \geq 2$; the general result is due to  Turner \cite{Turner}.

%\begin{definition}
We say that a group $G$ is a \emph{Turner group} if it satisfies the \emph{Retract Theorem}: an element $g\in G$ is a test element of $G$ if and only if it is not contained in any proper retract of $G$. 
%\end{definition}
Free groups of finite rank were the first examples of  Turner groups  (see \cite{Turner}).
Other examples are torsion free stably hyperbolic groups and finitely generated Fuchsian groups, as it was shown by O'Neill and Turner \cite{O'Neill}. Recently Groves \cite{Groves} proved that all torsion free hyperbolic groups are in fact Turner groups. 
The fundamental group of the Klein bottle $\langle a, b \mid aba^{-1}=b^{-1} \rangle $ is an example of a group which is not  a Turner group (see \cite{Turner}). However, all other surface groups are Turner groups.

\smallskip

In this paper we study test elements in the category of pro-$p$ groups, with emphasis on free pro-$p$ groups of finite rank and Demushkin groups. 
Furthermore, by relating test elements in finitely generated residually finite-$p$ Turner groups to test elements in their pro-$p$ completions, we obtain new results for test elements in free discrete groups and surface groups.
The starting point in our investigation is Theorem~\ref{test-retract}, which states that  every finitely generated profinite group is a Turner group.  
An immediate interesting consequence is that every element of infinite order in a finitely generated just infinite profinite group is a test element. In particular, all non-trivial elements in a torsion free just infinite pro-$p$ group are test elements.
This is in contrast to the case of non-procyclic free abelian pro-$p$ groups, which do not have any test elements.
\smallskip

%\begin{definition}
We write $\textrm{Aut}(G)$ for the group of (continuous) automorphisms of a profinite group $G$. The \emph{automorphic orbit} of an element $g \in G$ is defined by
\[\textrm{Orb}(g) = \{\alpha(g) \mid \alpha \in \textrm{Aut}(G) \}.\]
%\end{definition} 

In Corollary~\ref{orbit-free}, we prove that if $\varphi$ is an endomorphism of a free pro-$p$ group $F$ of finite rank such that $\varphi(\textrm{Orb}(u)) \subseteq \textrm{Orb}(u)$ for some $1 \neq u \in F$, then $\varphi$ is an automorphism. 
Moreover, this result holds for every finitely generated pro-$p$ group $G$ with the property that $\textrm{Aut}(G)$ acts transitively on $G/ {\Phi(G)}$ where $\Phi(G)$ denotes the Frattini subgroup of $G$ (Theorem \ref{orbit}).
The counterpart of Corollary~\ref{orbit-free} for free discrete groups was proved by Lee \cite{Lee2} (see also \cite{Lee1}). While Lee uses combinatorial techniques which are not available for free pro-$p$ groups, our proof is based on the Retract Theorem and uses 
``Frattini arguments".
\smallskip

The theory of free pro-$p$ groups although deprived of geometric techniques has one advantage vis-\'{a}-vis the discrete theory which is of crucial importance for our investigation, namely, every retract of a free pro-$p$ group is a free factor (see Corollary~\ref{test-free}).
This observation together with the basic fact that a subgroup $H$ of a free pro-$p$ group $F$ is a free factor of $F$ if and only if $\Phi(F) \cap H=\Phi(H)$ is already sufficient to establish numerous results on test elements in free pro-$p$ groups (see also Proposition~\ref{freefactors-intersection}). 

An \emph{almost primitive element} of a free pro-$p$ (or discrete) group $F$ is a non-primitive element $w \in F$  that 
is primitive in every proper subgroup of $F$ that contains it.  Another major impetus in our study of test elements of free pro-$p$ groups is the observation
that every almost primitive element of a free pro-$p$ group of finite rank is a test element (Proposition~\ref{almost primitive is test}). 
This is not true for free discrete groups; for example, $x_1[x_2, x_1]$ is an almost primitive element of $D(x_1, x_2)$ but it is not a test element.
Besides providing an important source of explicit examples of test elements in free pro-$p$ groups (see for example Proposition~\ref{almost_primitive - examples}), the notion of an almost primitive element is instrumental in the proof of the following result: the set of test elements in a non-abelian free pro-$p$ group $F$  is a dense subset of the Frattini subgroup of $F$ (Proposition~\ref{Frattini dense}).  

Sonn \cite{Sonn} proved that if $G$ is a Demushkin group with minimal generating set of size $n$ and $F$ is a free homomorphic image of $G$, then $F$ is of rank at most $\frac{n}{2}$. 
It follows that every retract of a Demushkin group is free pro-$p$ of rank at most $\frac{n}{2}$ (see Lemma~\ref{retract-Demushkin}). By combining this fact with our results on test elements in free pro-$p$ groups, we give many examples of test elements in Demushkin groups (Theorem~\ref{Demushkin-test} and Theorem~\ref{Demushkin-test-2}).
\smallskip

The connection between test elements in pro-$p$ groups and test elements in discrete groups is established in Proposition~\ref{pro-p - discrete}, which states that for a finitely generated residually finite-$p$ Turner group $G$, if $w \in G$ is a test element in some pro-$p$ completion of $G$, then it is a test element of $G$. We apply Proposition~\ref{pro-p - discrete} to our examples of test elements of free pro-$p$ groups (represented by words of finite length) and thus obtain many new examples of test elements of free discrete groups (see for example Proposition~\ref{examples - discrete}). Also, in this way we obtain all test elements in free discrete groups that we could find in the literature (including all of the examples mentioned above).  
Furthermore, we prove that the set of test elements in a free discrete group of finite rank is dense in the profinite topology (Theorem~\ref{dense - discrete}).

Let $G(n)=\langle x_1, \ldots, x_{2n} \mid [x_1, x_2] \ldots [x_{2n-1}, x_{2n}] \rangle$ be an orientable surface group of genus $n$.
Konieczny, Rosenberger and Wolny \cite{Konieczny} proved that for every prime $p$ and $n \geq 2$, $x_1^px_2^p \ldots x_{2n}^p$  is a test element of $G(n)$.
More generally, O'Neill and Turner \cite{O'Neill} proved that $x_1^kx_2^k \ldots x_{2n}^k$  is a test element of $G(n)$ for every $k \geq 2$ and $n \geq 2$.
To the best of our knowledge, these are the only known examples of test elements in surface groups. 
Since a pro-$p$ completion of an orientable surface group is a Demushkin group, using Proposition~\ref{pro-p - discrete} and our results on test elements of Demushkin groups,  we obtain many new examples of test elements in orientable surface groups
(Proposition~\ref{surface - test} and Proposition~\ref{surface - test - 2}).  Moreover,  if $G$ is a non-orientable surface group of genus $n \geq 3$, then $G$ is residually finite-$p$  Turner group, and for $p \geq 3$, the pro-$p$ completion of $G$ is a free pro-$p$ group
of rank $n-1$. From this observation we derive multitude of examples of test elements of non-orientable surface groups (Proposition~\ref{non-orientable surface groups}).

\subsection{Notation}
Throughout the paper, $p$ denotes a prime. The $p$-adic integers are
denoted by $\Z_p$.
Subgroup $H$ of a profinite group $G$ is tacitly taken to be
closed and by generators we mean topological generators as
appropriate. If $G$ is a profinite group and $g_1, \ldots, g_n \in G$, then $\langle g_1, \ldots, g_n \rangle$ is the closed subgroup of $G$ topologically generated by $g_1, \ldots, g_n$.  
The minimal number of generators of a profinite group $G$ is denoted by $d(G)$. Homomorphisms between profinite groups are always assumed to be continuous.

We write $\Phi(G)$ for the Frattini subgroup of a pro-$p$ group $G$. 
We use the notation $G=H \coprod K$ to express that the pro-$p$ group $G$ is a free product of the groups $H$ and $K$; 
$H \leq_{\textrm{ff}} G$ means that $H$ is a free factor of $G$.  
\section{Subgroup properties}

In this section we establish some preliminary results which extend to the categories of profinite and pro-$p$ groups some well-known properties of free discrete groups.

A family of sets $\{X_i \mid i \in I\}$ is  \emph{filtered from below} if for all $i_1, i_2 \in I$, there is $j \in I$ for which $X_j \subseteq X_{i_1} \cap X_{i_2}$.

\begin{pro}\label{descendin-chain-rank}
Let $G$ be a profinite group and $\mathcal{H}=\{H_i \mid i \in I\}$ be a family of subgroups of $G$ filtered from below. 
Let $H=\bigcap_{i \in I}H_i$.
\begin{itemize}
\item [(a)] If there exists a natural number $r$ such that $d(H_i) \leq r$ for all $i \in I$, then $d(H) \leq r$. Moreover, if  $G$ is a pro-$p$ group and $(\mathcal{H}, \subseteq)$ does not contain a smallest element, then $d(H) < r$.  
\item [(b)] If $G$ is a free pro-$p$ group and $K$ is a finitely generated free factor of $H$, then there exists $i_0 \in I$  with the property that $K$ is a free factor of $H_i$ for all $H_i \leq H_{i_0}$.
\end{itemize} 
\end{pro}
\begin{proof}
$(a)$ Suppose that there is  a natural number $r$ such that $d(H_i) \leq r$ for all $i \in I$. Assume to the contrary that $d(H) > r$. Then there exists an epimorphism $\varphi: H \twoheadrightarrow T$ onto a finite group $T$ that can not be generated by $r$ elements. 
Let $M\trianglelefteq_o G$ with $H\cap M \leq \textrm{ker}(\varphi)$. Using that $\mathcal{H}$ is a family of closed subgroups filtered from below, we get 
\begin{displaymath}
HM =  (\bigcap_{i\in I}H_i) M = \bigcap_{i\in I}{H_iM}.
\end{displaymath}
 Since $HM$ is open in $G$, $HM=H_{i_0}M$ for some $i_0 \in I$.  It follows that the composition of the following homomorphisms 
 \begin{displaymath}
 H_{i_0} \twoheadrightarrow H_{i_0}/{(H_{i_0}\cap M)} \cong {H_{i_0}M}/M \cong {HM}/ M \cong H/{(H\cap M)}  {\mathop{\twoheadrightarrow}^{\mathrm{\bar{\varphi}}}} T
 \end{displaymath}
 where $\bar{\varphi}$ is defined by $\bar{\varphi}(x(H \cap M))=\phi(x)$ is an epimorphism.  This is a contradiction with $d(H_{i_0}) \leq r$. We will return to the second statement in  $(a)$ after we establish $(b)$. \\ 
%\begin{displaymath}
 %\xymatrix{    
 %H \ar@{->>}[d] \ar@{->>}[r]^{\varphi} & T \\
 %H/{(H \cap M)} \ar@{->>}[ur]_{\bar{\varphi}} &  }
 %\end{displaymath}
 $(b)$ Suppose that $G$ is a pro-$p$ group and that $K$ is a finitely generated free factor of $H$. Then $\Phi(H) \cap K = \Phi(K)$. Since $\mathcal{H}$ is filtered from below, we have 
 \begin{displaymath}
 \Phi(H) = \bigcap_{i \in I}\Phi(H_i) \text{ and } \Phi(K)=\Phi(H) \cap K= \bigcap_{i \in I}\Phi(H_i) \cap K.
 \end{displaymath}

As $K$ is finitely generated, it follows that $\Phi(K)$ is open in $K$. Hence, there exists $i_0 \in I$  such that for all $H_i \leq H_{i_0}$,  $\Phi(H_i) \cap K = \Phi(K)$. 
In the case when $G$ is a free pro-$p$ group, it follows that $K$ is a free factor of $H_i$ for all $H_i \leq H_{i_0}$.   
Furthermore, if $G$ is any pro-$p$ group and  there is $r$ for which $d(H_i) \leq r$ for all $i \in I$, then by setting $K=H$ (as we already proved, $H$ is finitely generated in this case), we get $\Phi(H_{i_0}) \cap H = \Phi(H)$.
Hence, there is an embedding $H / \Phi(H) \hookrightarrow H_{i_0}/\Phi(H_{i_0})$, which is an isomorphism if and only if $H=H_{i_0}$. 
This proves the second statement in $(a)$.
\end{proof}

\begin{pro}
\label{freefactors-intersection}
Let $\{ H_i  \}_{i\in I}$ and $\{ K_i  \}_{i\in I}$ be two families of subgroups of a free pro-$p$ group $F$. If $H_i$ is a free factor of $K_i$ for all $i\in I$, then $\bigcap_{i\in I}H_i$ is a free factor of $ \bigcap_{i\in I}K_i$.
\end{pro}
\begin{proof}

We begin the proof with three auxiliary claims.
\smallskip

\noindent \emph{Claim 1.} Let $H$ and $K$ be two subgroups of $F$  satisfying $H\leq_{\textrm{ff}} K$. Then $H\cap M \leq_{\textrm{ff}} K\cap M$ for every $M\leq F$. 

\smallskip

\noindent \emph{Proof of Claim 1.} Let $\{ U_j \mid j \in J \}$ be the set of all open subgroups of $K$ that contain $K\cap M$. Then $ K\cap M = \bigcap_{j \in J} U_j$ and $ \Phi(K\cap  M) = \bigcap_{j \in J} {\Phi(U_j)}$. By the Kurosh subgroup theorem, $U_j\cap H$ is a free factor of $U_j$. Hence, $$ \Phi(U_j\cap H) = \Phi(U_j) \cap U_j \cap H = \Phi(U_j) \cap H$$ for all $j \in J$. Furthermore,
\begin{displaymath}
\Phi(K\cap M) \cap H = (\bigcap_{j \in J}\Phi(U_j))\cap H = \bigcap_{j \in J}{\Phi(U_j\cap H)} = \Phi(K\cap M \cap H) = \Phi(H\cap M)
\end{displaymath}
 and $\Phi(K\cap M)\cap ( H  \cap M)= \Phi(K\cap M) \cap H = \Phi(H\cap M)$. Hence, $H\cap M \leq_{\textrm{ff}} K\cap M$.

\smallskip

\noindent \emph{Claim 2.}  Let $H, K,$ and $M$ be subgroups of $F$ with $H \leq_{\textrm{ff}} K \leq_{\textrm{ff}} M$. Then  $H \leq_{\textrm{ff}} M$.

\smallskip

\noindent \emph{Proof of Claim 2.}   From $\Phi(K) = K \cap \Phi(M)$ and $\Phi(H) = H \cap \Phi(K)$, we get
 \begin{displaymath}
\Phi(M) \cap H = \Phi(M) \cap K \cap H = \Phi(K) \cap H = \Phi(H).
 \end{displaymath}

\smallskip

\noindent \emph{Claim 3.} Let $H_1, H_2, K_1,$ and $K_2$ be subgroups of $F$ with $H_1  \leq_{\textrm{ff}} K_1$ and $H_2  \leq_{\textrm{ff}} K_2$. Then $H_1 \cap H_2  \leq_{\textrm{ff}} K_1 \cap K_2$.
\smallskip

\noindent \emph{Proof of Claim 3.}  By Claim 1, $H_1 \cap H_2 \leq_{\textrm{ff}} K_1 \cap H_2$ and $K_1 \cap H_2 \leq_{\textrm{ff}} K_1 \cap K_2$. It follows from Claim 2 that
$H_1 \cap H_2  \leq_{\textrm{ff}} K_1 \cap K_2$.

\smallskip

Now we are ready to prove the proposition. Set $H = \bigcap_{i \in I}H_i$ and $K = \bigcap_{i \in I}K_i$. 
Let $\mathcal{I}$ be the set of all finite subsets of $I$. 
For each $A \in \mathcal{I}$, let $H_A=\bigcap_{i \in A} H_i$ and $K_A=\bigcap_{i \in A} K_i$.
It follows by induction from Claim 3 that $H_A \leq_{\textrm{ff}} K_A$ for all $A \in \mathcal{I}$. Observe that $\{H_A \mid A \in \mathcal{I}\}$ and $\{K_A \mid A \in \mathcal{I}\}$ are families of subgroups of $F$ filtered from below. Moreover, 
$\bigcap_{A \in \mathcal{I}}H_A=H \text{ and } \bigcap_{A \in \mathcal{I}}K_A=K.$
Hence, $\Phi(H)=\bigcap_{A \in \mathcal{I}}\Phi(H_A) \text{ and } \Phi(K)=\bigcap_{A \in \mathcal{I}}\Phi(K_A).$ 
Consequently,
\[\Phi(K) \cap H=(\bigcap_{A \in \mathcal{I}}\Phi(K_A)) \cap (\bigcap_{A \in \mathcal{I}} H_A)=\bigcap_{A \in \mathcal{I}}\Phi(K_A) \cap H_A=\bigcap_{A \in \mathcal{I}}\Phi(H_A)=\Phi(H),\]
and thus $H \leq_{\textrm{ff}} K$.
\end{proof}

\begin{definition}
Let $F$ be a free pro-$p$ group and $K \leq F$. Let $\{H_i \mid i \in I\}$ be the set of all free factors of $F$ that contain $K$.
We call the subgroup $A=\bigcap_{i \in I} H_i$ the algebraic closure of $K$ in F.
Note that by Proposition~\ref{freefactors-intersection}, $A$ is a free factor of $F$.
\end{definition}

The term ``algebraic closure" is well established in the theory of free discrete groups (see \cite{AlgEx}).
The concepts of algebraic closure and  algebraic extension of subgroups of a free pro-$p$ group will be developed further in a subsequent paper by the authors.

\iffalse
Recall that an element of a free pro-$p$ group $F$ that is a member of a basis of $F$ is called a primitive element of $F$.

\begin{pro}\label{power-general} 
Let $x$ be an element in a free pro-$p$ group $F$, and let $H$ be a subgroup of $F$ satisfying $x, x^p, ... , x^{p^{n-1}} \notin H$ but $x^{p^n} \in H$ for some $n\geq 1$.  Then $x^{p^n}$ is a primitive element of $H$. 
\end{pro}
\begin{proof}
Note that there exists $U \leq_o F$ with $H \leq U$ and $x, x^p, ... , x^{p^{n-1}} \notin U$. Let $K=\overline{\langle \{x\} \cup U \rangle}$. Then $x$ is a primitive element of $K$. Indeed, the assumption that $x$ is not primitive in $K$ implies 
$K=U$, which contradicts $x \notin U$. Now $K = \overline{\langle x \rangle} \amalg  C$ for some $C\leq K$. Since $U$ is an open subgroup of $K$, the Kurosh subgroup theorem applies, and thus $U = {U \cap \overline{\langle x \rangle}} \amalg  T$ for some $T \leq U$. 
Since $U \cap \overline{\langle x \rangle}$ is clearly generated by $x^{p^n}$, it follows that $x^{p^n}$ is a primitive element of $U$. As $H \leq U$, $x^{p^n}$  is also a primitive element of $H$.
\end{proof}
\fi
A subgroup $H$ of a pro-$p$ group $G$ is \emph{isolated} in $G$ if for every $g \in G$ and $1 \neq \alpha \in \mathbb{Z}_p$, 
$g^{\alpha} \in H$  implies $g \in H$. 

\begin{pro}
\label{isolated}
Let $F$ be a free pro-$p$ group and $H \leq_{ \textrm{ff} } F$. Then $H$ is isolated in $F$.
\end{pro}
\begin{proof}
Let $F=H \coprod K$, and define $r:F \to H$ by $r(h)=h$ for all $h \in H$ and $r(k)=1$ for all $k \in K$.
Suppose that $x^{\alpha} \in H$ for some $x \in F$ and $1 \neq \alpha \in \mathbb{Z}_p$.
Then $x^{\alpha}=r(x^{\alpha})=r(x)^{\alpha}$ and by unique extraction of roots in free pro-$p$ groups, $x=r(x) \in H$. 
\end{proof}
\medskip

\section{Test elements of profinite groups}

The following concept appears in \cite{Turner}, where it was used in the proof of the Retract Theorem for free discrete groups of finite rank. 
\begin{definition}
Let $\varphi$ be an endomorphism of a group $G$. The stable image of $\varphi$ is defined by
\[\varphi^{\infty}(G) = \bigcap_{n=1}^{\infty}\varphi^n(G).\]
\end{definition}

\begin{lem}
\label{varhi_infinity}
Let $\varphi$ be an endomorphism of a profinite group $G$. The restriction of $\varphi$ to $\varphi^{\infty}(G)$ gives rise to an epimorphism $\varphi_{\infty}:\varphi^{\infty}(G) \to \varphi^{\infty}(G)$.
Moreover, if $G$ is finitely generated, then $\varphi_{\infty}$ is an isomorphism.
\end{lem}
\begin{proof}
Note that $\varphi^{n}(G) \supseteq \varphi^{n+1}(G)$ for all $n \geq 1$. Consequently,
\begin{displaymath}
\varphi(\varphi^{\infty}(G)) = \varphi(\bigcap_{n=1}^{\infty}\varphi^n(G)) \subseteq \bigcap_{n=1}^{\infty}\varphi^{n+1}(G) =\varphi^{\infty}(G).
\end{displaymath}
Hence,  $\varphi_{\infty}(g):=\varphi(g)$ defines an endomorphism of $\varphi^{\infty}(G)$. 

Let $x \in  \varphi^{\infty}(G)$ and set $X_n = ({{\varphi}^{n}})^{-1}(x)$ for $n\geq 1$. There is an obvious inverse system $(X_n, \varphi_{n,m})$ over the positive integers where for $n \geq m$, $\varphi_{n,m}:X_n \to X_m$ is defined by $\varphi_{n,m}(g)=\varphi^{n-m}(g)$ for all $g \in X_n$ . Since all $X_n$ are clearly compact, $\varprojlim X_n \neq \emptyset$. Let ${(a_n)}_{n=1}^{\infty} \in \varprojlim X_n $. Then $\varphi(a_1) = x$ and $a_1 \in \varphi^{\infty}(G)$ ($\varphi^n(a_{n+1})=\varphi_{n+1,1}(a_{n+1}) = a_1$ implies $a_1 \in \varphi^n(G)$ ). Hence, 
$\varphi_{\infty}$ is surjective. 

If $G$ is finitely generated with $d(G)=r$, then $d(\varphi^n(G)) \leq r$ for all $n \geq 1$. By Proposition~\ref{descendin-chain-rank} $(a)$, $d(\varphi^{\infty}(G)) \leq r$ and it follows from the Hopfian property of finitely generated 
profinite groups that $\varphi_{\infty}$ is an isomorphism. 
\end{proof}

\begin{rem}
If $G$ is not finitely generated, then $\varphi_{\infty}$ may not be injective. Indeed, let $F$ be the free pro-$p$ group on the infinite set $X=\{x_i \mid i \in \mathbb{N}\}$. Define $\varphi:F \to F$ by $\varphi(x_0)=1$, 
$\varphi(x_{2^n})=x_0$ for $n \in \mathbb{N}$, and $\varphi(x_m)=x_{m+1}$ for $m \in \mathbb{N} \setminus (\{2^n \mid n \in \mathbb{N}\} \cup \{0\})$. Then $x_0 \in \varphi^{\infty}(F)$ but $\varphi_{\infty}(x_0)=1$.   
\end{rem}

Recall that a \emph{retract} of a group $G$ is a subgroup $H \leq G$ for which there exists an epimorphism $r:G \to H$ that restricts to the identity homomorphism on $H$; such epimorphism $r$ is called \emph{retraction}.
\begin{lem}
\label{varphi-retract}
Let $\varphi$ be an endomorphism of a finitely generated profinite group $G$. Then $\varphi^{\infty}(G)$ is a retract of $G$.
\end{lem}
\begin{proof}  
By Proposition~\ref{descendin-chain-rank} $(a)$, $\varphi^{\infty}(G)$ is finitely generated. Let $\{g_1, ..., g_r  \}$ be a generating set for  $\varphi^{\infty}(G)$, and
extend $\{ g_1, ..., g_r \}$  to a generating set $\{ g_1, ..., g_m \}$ ($m\geq r$) for $G$.  
We define a homomorphism $\alpha: G \to \varphi^{\infty}(G)$ as follows. Consider the sequence ${\{ (\varphi^k(g_1), ... , \varphi^k(g_m)) \}}_{k=1}^{\infty}$ of elements 
 in $G \times \cdots \times G$. After possibly passing to a subsequence, we can assume that $\displaystyle \lim_{k\to \infty}{ (\varphi^k(g_1), ... , \varphi^k(g_m) )} = (h_1, ..., h_m)$ exists. 
Given $ n \geq 1$ and $1 \leq i \leq m$, $\varphi^k(g_i) \in \varphi^{k}(G) \subseteq \varphi^{n}(G)$ for all $k \geq n$. Hence, $\displaystyle h_i = \lim_{k \to \infty}\varphi^k(g_i) \in \varphi^{n}(G)$, and since $n$ is arbitrary, $h_i \in \varphi^{\infty}(G)$. 
Let $u(x_1, ... , x_m)$ be any word in $m$ variables (i.e. $u(x_1, ... , x_m)$ is an element in the free profinite group on $\{x_1, \ldots, x_m\}$). Then
 \begin{displaymath}
 u(h_1, ... , h_m) = u(\lim_{k\to \infty}\varphi^k(g_1), ..., \lim_{k \to \infty}\varphi^k(g_m)) = 
 \end{displaymath}
 \begin{displaymath}
 \lim_{k\to \infty}u(\varphi^k(g_1), ..., \varphi^k(g_m)) = \lim_{k \to \infty}\varphi^k(u(g_1, \cdots, g_m)).
 \end{displaymath}
Therefore, $u(g_1, ..., g_m) = 1$ implies $u(h_1, ..., h_m) =1$. This shows that the assignment $\alpha(g_i) = h_i$, $1 \leq i \leq m$, defines a continuous homomorphism $\alpha: G \to \varphi^{\infty}(G)$. 

By  Lemma~\ref{varhi_infinity}, the restriction of $\varphi$ to $\varphi^{\infty}(G)$ yields an automorphism $\varphi_{\infty}:\varphi^{\infty}(G) \to \varphi^{\infty}(G)$.
Since $\phi^{\infty}(G)$ is finitely generated, $\textrm{Aut}(\varphi^{\infty}(G))$ is a profinite group. It follows that after possibly passing to a subsequence,  we can assume that 
$\displaystyle \lim_{k \to \infty}\varphi_{\infty}^k= \beta$  exists. By definition of the congruence subgroup topology on $\textrm{Aut}(\varphi^{\infty}(G))$, for every $N \unlhd \varphi^{\infty}(G)$, there exists $n \geq 1$ such that 
$\beta^{-1}(\varphi_{\infty}^k(g))g^{-1} \in \beta^{-1}(N)$ (or equivalently $\varphi_{\infty}^k(g)\beta(g)^{-1} \in N$) for all $k \geq n$ and $g \in \varphi^{\infty}(G)$. 
In particular,  $\varphi_{\infty}^k(g_i)\beta(g_i)^{-1} \in N$ for all $1 \leq i \leq r$; thus $\displaystyle \alpha(g_i)\beta(g_i)^{-1}=\lim_{k \to \infty}\varphi_{\infty}^k(g_i)\beta(g_i) \in N$ for $1 \leq i \leq r$. 
Since $g_1, \ldots, g_r$ generate $\varphi^{\infty}(G)$ and $N \unlhd \varphi^{\infty}(G)$ is an arbitrary normal subgroup, it follows that $\beta=\alpha_{|\varphi^{\infty}(G)}$. 
Therefore, $\alpha_{|\varphi^{\infty}(G)}$ is an automorphism and  $\alpha_{|\varphi^{\infty}(G)}^{-1}\circ \alpha: G \to \varphi^{\infty}(G) $  is a retraction.
\end{proof}

\begin{thm}
\label{test-retract}
The test elements of a finitely generated  profinite group $G$ are exactly the elements not contained in any proper retract of $G$.
\end{thm}
\begin{proof}
Let $g \in G$. If $g$ belongs to a proper retract $H$ of $G$, then the composition of a retraction $r: G \to H$ and the inclusion homomorphism $H \hookrightarrow G$ is an endomorphism of $G$ which fixes $g$ but is not an automorphism. 

Now assume that $g \in G$ is not a test element of $G$. Then there exists a homomorphism $\varphi: G \to G$ that is not an automorphism and such that $\varphi(g) = g$. Since $G$ is finitely generated, $\varphi$ is not surjective and thus  $\varphi^{\infty}(G)$ is a proper subgroup of $G$. By Lemma~\ref{varphi-retract}, $\varphi^{\infty}(G)$ is a retract of $G$. Clearly, $g \in \varphi^{\infty}(G)$.
\end{proof}

\begin{cor}
\label{test-free}
The only retracts of a free pro-$p$ group are the free factors.
The test elements of a free pro-$p$ group $F$ of finite rank are the elements not contained in any proper free factor of $F$.
\end{cor}
\begin{proof}
Let $F$ be a free pro-$p$ group and $r:F \to H$ be a retraction. Then $\Phi(H) \subseteq H \cap \Phi(F)$, and
using that $r$ is a retraction, we get
\[H \cap \Phi(F)=r(H \cap \Phi(F)) \subseteq r(H) \cap r(\Phi(F))=H \cap \Phi(H)=\Phi(H).\]
Hence, $H \cap \Phi(F)=\Phi(H)$ and $H$ is a free factor of $F$.
The last statement of the Corollary follows from Theorem~\ref{test-retract}.
\end{proof}

We can paraphrase the second statement in Corollary~\ref{test-free} as follows. An element $w$ in a free pro-$p$ group $F$ of finite rank is a test element of $F$ if and only if $F$ is algebraic over $\langle w \rangle$, i.e, 
$F$ is the algebraic closure of $\langle w \rangle$ in $F$.

\smallskip

Recall that a profinite group $G$ is termed \emph{just infinite} if $G$ is infinite and every proper quotient of $G$ is finite. The groups $\textrm{SL}_d(\mathbb{Z}_p)$ and $ \textrm{SL}_d(\mathbb{F}_p[[t]])$, where $\mathbb{F}_p[[t]]$ is the pro-$p$ ring of formal power series over a finite field with $p$ elements, are examples of just infinite profinite  groups. Note that these groups are virtually pro-$p$. Indeed,  many of the well-known just infinite pro-$p$ groups are groups of Lie type defined over $\mathbb{Z}_p$ (i.e., $p$-adic analytic groups) or over $\mathbb{F}_p[[t]]$.   In particular,  $\textrm{SL}_d^1(\mathbb{Z}_p) = \textrm{ker}(\textrm{SL}_d(\mathbb{Z}_p) \to \textrm{SL}_d(\mathbb{Z}_p/p{\mathbb{Z}_p})) $, the first congruence subgroup of $\textrm{SL}_d(\mathbb{Z}_p)$, is a just infinite pro-$p$ group.  For a thorough discussion of $p$-adic analytic just infinite pro-$p$ groups see \cite{Klaas}. An example of a just infinite pro-$p$ group linear over $\mathbb{F}_p[[t]]$ is $\textrm{SL}_d^1(\mathbb{F}_p[[t]]) = \textrm{ker}(\textrm{SL}_d(\mathbb{F}_p[[t]]) \to \textrm{SL}_d(\mathbb{F}_p[[t]]/{t\mathbb{F}_p[[t]]})) $, the first congruence subgroup of $ \textrm{SL}_d(\mathbb{F}_p[[t]])$. The Nottingham group  $\mathcal{N}(\mathbb{F}_p)$, which can be described as  the group of automorphisms of $\mathbb{F}_p[[t]]$ that act trivially on $t\mathbb{F}_p[[t]]/{t^2\mathbb{F}_p[[t]]}$, is an example of a just infinite pro-$p$ group that is not linear over any profinite ring. The most remarkable property of $\mathcal{N}(\mathbb{F}_p)$  is the fact that it is universal, namely, it contains a copy of every countably based pro-$p$ group. A careful treatment of just infinite profinite groups is given in \cite{Wilson}.

\begin{cor}
In a finitely generated just infinite profinite group every element of infinite order is a test element. In particular, in a torsion free finitely generated just infinite profinite group every non-trivial element is a test element.
\end{cor}
\begin{proof}
Let $G$ be a finitely generated just infinite profinite group, and let $w\in G$ be an element of infinite order. Suppose that $w$ is not a test element of $G$. By Theorem \ref{test-retract}, there is a proper retract $H$ of $G$ containing $w$. Thus $H$ is an infinite proper quotient of $G$, which is a contradiction. Hence, $w$ is a test element of $G$.
\end{proof}

Note that if $G$ is a just infinite pro-$p$ group, then it is finitely generated since $G/{\Phi(G)}$ is finite. Thus the above result holds for all just infinite pro-$p$ groups.

A profinite group $G$ is hereditarily  just infinite if every open subgroup of $G$ is just infinite. All the just infinite pro-$p$ (profinite) groups mentioned above are indeed examples of hereditarily just infinite pro-$p$ (profinite) groups.

\begin{cor}
Let $G$ be a finitely generated hereditarily just infinite profinite group, and let $w\in G$ be an element of infinite order. Then $w$ is a test element of every open subgroup of $G$ that contains it.
\end{cor}

Contrary to the case of just infinite  torsion free pro-$p$ groups where every non-trivial element is a test element, free abelian pro-$p$ groups of finite rank greater than one do not have test elements.
To see this,  let  $G = \mathbb{Z}_p^n$ be the free abelian pro-$p$ group of rank $n \geq 2$ and suppose that $x \in \Phi(G)=G^p$. Then there is a positive integer $k$ such that $x\in G^{p^k} \setminus G^{p^{k+1}}$.  By unique extraction of roots, there exists $y \in G$ such that $y^{p^k} = x$. Clearly, $y \in G \setminus G^p$ and so $\langle y \rangle$ is a proper retract of $G$. Hence, by Theorem  \ref{test-retract}, $x$ is  not a test element of $G$.

\iffalse
We digress from our main topic to note an interesting consequence of the results proved in this section. Recall that the fixed point set of an endomorphism $\varphi$ of a group $G$ is 
\[Fix(\varphi)=\{g \in G \mid \varphi(g)=g\}.\]
\begin{pro} 
Let $\varphi$ be an endomorphism of a free pro-$p$ group $F$ of finite rank. Then there is an automorphism $\alpha: F \to F$ with $Fix(\varphi)=Fix(\alpha)$. 
\end{pro}
\begin{proof}
By Lemma~\ref{varphi-retract}, $\varphi^{\infty}(F)$ is a retract of $F$, and it follows from Corollary~\ref{test-free} that $\varphi^{\infty}(F)$ is a free factor of $F$.
Let $F=\varphi^{\infty}(F) \coprod H$ and define a homomorphism $\alpha:F \to F$ by $\alpha(w)=\varphi(w)$  
\end{proof}
\fi

\section{Endomorphisms that preserve an automorphic orbit}
 
\begin{lem}
\label{retract-intersection}
Let $G$ be a profinite group, and let $H \leq K \leq G$.
\begin{itemize}
\item[(a)]  If $K$ is a retract of $G$ and $H$ is a retract of $K$, then $H$ is a retract of $G$.
\item[(b)]  If $H$ is a retract of $G$, then it is also a retract of $K$. 
\item[(c)] If $G$ is finitely generated, then $H$ is contained in a retract of $G$ that is minimal among retracts of $G$ that contain $H$.
Moreover, all such minimal retracts over $H$ are isomorphic.
\end{itemize} 
\end{lem}
\begin{proof}
Parts $(a)$ and $(b)$ are trivial. Suppose that $G$ is finitely generated, and  
let $\{R_i\}_{i \in I}$ be a chain of retracts of $G$ all of which contain $H$. Set $ R=\bigcap_{i \in I}R_i$.
We will show that there is a descending sequence 
\[R_{i_0} \geq R_{i_1} \geq R_{i_2} \geq \ldots \] 
with $R=\bigcap_{k=0}^{\infty}R_{i_k}$. For each $i \in I$, let $\mathcal{S}_i$ be the set of all open subgroups of $G$ that contain $R_i$. Then $\mathcal{S}=\bigcup_{i\in I}\mathcal{S}_i$ is the set of open subgroups of $G$ that
contain $R$. As $G$ is finitely generated, it has countably many open subgroups. Therefore, $\mathcal{S}=\{U_1, U_2, \ldots\}$ is a counable set.
Let $R_{i_0}$ be any subgroup in $\{R_i\}_{i \in I}$ contained in $U_1$, and let $m$ be the smallest natural number such that $U_m \notin \mathcal{S}_{i_0}$. (If no such $m$ exists, then set $R_{i_k}=R_{i_0}$ for all $k \geq 1$). 
Choose $R_{i_1}$ to be any subgroup in $\{R_i\}_{i \in I}$ contained in $U_m$.  Since  $\{R_i\}_{i \in I}$ is a chain, $R_{i_0} \geq R_{i_1}$. It is now clear that if we continue this procedure, we will obtain the desired descending sequence.

For each $k \in \mathbb{N}$, let $r_{i_k}:G \to R_{i_k}$ be a retraction.
Clearly, $d(R_{i_k}) \leq d(G)$ for all $k \in \mathbb{N}$.
By Proposition~\ref{descendin-chain-rank} $(a)$, $R$ is finitely generated. 
Let $\{g_1, \ldots, g_l\} \subseteq R$ be a generating set for $R$ and $\{g_1, \ldots, g_n\} \subseteq G$ $ (l \leq n)$ be a generating set for $G$.
As in the proof of Lemma~\ref{varphi-retract}, after possibly passing to a subsequence, we can assume that $\displaystyle \lim_{k\to \infty}{ (r_{i_k}(g_1), ... , r_{i_k}(g_n) )} = (h_1, ..., h_n)$ for some $\{h_1, \ldots, h_n\} \subseteq R$.
Again as in the proof of Lemma~\ref{varphi-retract}, we have a well-defined homomorphism 
$r:G \to R$ satisfying $r(g_j)=h_j$ for all $1 \leq j \leq n$.  Since for all $k \in \mathbb{N}$, $r_{i_k}$ restricts to the identity on $R$, we have $h_j=g_j$ for $1 \leq j \leq l$. Hence, $r$ is a retraction.
Now the first statement in $(c)$ follows from Zorn's lemma. 

Suppose that $ T_1$ and $T_2$ are two retracts of $G$ that contain $H$ and are both minimal among retracts of $G$ that contain $H$.
Let $r_1:G \to T_1$ and $r_2:G \to T_2$ be retractions. Since $r_2(T_1) \subseteq T_2$, we can restrict $r_2$ to obtain a homomorphism $r'_2:T_1 \to T_2$.
Similarly, by restricting $r_1$, we get a homomorphism $r'_1:T_2 \to T_1$. Consider the homomorphism $\alpha=r'_1 \circ r'_2:T_1 \to T_1$.
As $H \subseteq T_1 \cap T_2$, for every $h \in H$, we have 
\[\alpha(h)=r'_1(r'_2(h))=r_1(r_2(h))=r_1(h)=h.\]
So, $H \subseteq Fix(\alpha) \subseteq \alpha^{\infty}(T_1)$. Since $G$ is finitely generated, so is $T_1$, and it follows from Lemma~\ref{varphi-retract} that $\alpha^{\infty}(T_1)$ is a retract of $T_1$.
Moreover, by part $(b)$, $\alpha^{\infty}(T_1)$ is a retract of $G$. The minimality condition on $T_1$ yields $T_1=\alpha^{\infty}(T_1)$.
Therefore, $\alpha(T_1)=T_1$, and by the Hopfian property of finitely generated profinite groups, $\alpha$ is an isomorphism.
Analogously, $r'_2 \circ r'_1: T_2 \to T_2$ is an isomorphism. Hence, $r'_1:T_1 \to T_2$ and $r'_2:T_2 \to T_1$ are isomorphisms.  
\end{proof}

\begin{rem}
It follows from Corollary~\ref{test-free} and Proposition~\ref{freefactors-intersection} that for every subgroup $H$ of a free pro-$p$ group $F$, there is a unique minimal retract of $F$ over $H$, namely, the algebraic closure of $H$ in $F$. One might wonder whether  for every ``nice" subgroup $H$  (maybe every finitely generated subgroup if not all subgroups)  of a ``nice"" pro-$p$ group $G$ there is always a unique minimal retract of $G$ over $H$. 
That can hardly be the case since, as we shell see, it already fails for Demushkin groups (see Section~\ref{demushkin} for the definition of a Demushkin group).
Let $G$ be a pro-$p$ completion of an orientable surface group of  even genus, i.e.,   $G = \langle x_1, x_2, \ldots, x_{2n} \mid [x_1, x_2][x_3, x_4] \ldots [x_{2n-1}, x_{2n}] \rangle$ for some even positive integer $n$.
Set $H_1 = \langle x_1, \ldots, x_n \rangle, H_2=\langle x_{n+1}, \ldots, x_{2n} \rangle$, and  $A=\langle [x_1, x_2] \ldots [x_{n-1}, x_n] \rangle=\langle [x_{n+1}, x_{n+2}] \ldots [x_{2n-1}, x_{2n}] \rangle$. 
Define  $r_1:G \to H_1$ by $r_1(x_i) = x_i$ for $1 \leq i \leq n$ and $r_1(x_{n+i+1})=x_{n - i}$ for $0 \leq i \leq n-1$. Clearly, $r_1$ is a well-defined retraction onto $H_1$.
Similarly, we have a retraction $r_2:G \to H_2$ defined by $r_2(x_{i+1}) = x_{2n-i}$ for $0 \leq i \leq n-1$ and $r_2(x_{i})=x_{ i}$ for $n+1 \leq i \leq 2n$.
Note that $G=H_1 \coprod_A H_2$ is a proper amalgam and $H_1 \cap H_2 = A$. Clearly, $H_1$ is a free pro-$p$ group  on $\{x_1, \ldots, x_n\}$ (see Section~\ref{demushkin}). It follows from Proposition~\ref{almost_primitive - examples} $(d)$ that
 $[x_1, x_2] \ldots [x_{n-1}, x_n] \in H_1$ is a test element of $H_1$, and thus by Theorem~\ref{test-retract}, $A$ is not contained in a proper retract of $H_1$. Therefore, $H_1$ is a minimal retract of $G$ over $A$.
A similar argument shows that $H_2$ is also a minimal retract of $G$ over $A$. Note that this example also shows that the intersection of two retracts of $G$ is not necessarily a retract of $G$. 
In fact, it is easy to see that for a finitely generated profinite group $G$ the intersection of retracts is always a retract if and only if every subgroup $H \leq G$ is contained in a unique minimal retract of $G$ over $H$.
%However, if $R_1$ and $R_2$ are any two retracts of a Demushkin group $G$, then $d(R_i) \leq n$ for $i=1,2$ (see Section~\ref{demushkin}); in the case when $d(R_1)+d(R_2) < 2n$ and $d(G)>2$, it is readily seen that $H=\overline{\langle R_1 \cup R_2 \rangle}$ is a 
%free pro-$p$ group. By Lemma~\ref{retract-intersection} (b), $R_1$ and $R_2$ are retracts of $H$, and thus $R_1 \cap R_2$ is a retract of $R_1$. It follows from Lemma~\ref{retract-intersection} (a) that $R_1 \cap R_2$ is a retract of $G$.
%At the end of this rather long remark we would like to admit that we do not know if there exist Demuskin groups, besides the obvious ones, for which the intersection of two retracts is always a retract. (make this precise!)
\end{rem}

An element of a pro-$p$ group $G$ is called a \emph{primitive element} of $G$ if it is a member of a minimal generating set for $G$.
In other words, the primitive elements of $G$ are exactly the elements in $G \setminus \Phi(G)$. 

\begin{lem}\label{orbit-primitive}
Let $G$ be a finitely generated pro-p group and $\varphi:G \to G$ be a homomorphism. If $\varphi$ maps primitive elements to primitive elements, then it is an isomorphism. 
\end{lem}
\begin{proof}
Suppose that $\varphi$ maps primitive elements to primitive elements. So, $x \notin \Phi(G)$ implies $\varphi(x) \notin \Phi(G)$.
Therefore, the induced homomorphism $\bar{\varphi}: G/{\Phi(G)} \to G/{\Phi(G)}, \bar{\varphi}(x\Phi(G))=\varphi(x) \Phi(G)$, is injective.  
Since $G/{\Phi(G)}$ is a finite dimensional $\mathbb{F}_p$-vector space, it  follows that $\bar{\varphi}$  is an isomorphism. Hence, $\varphi$ is an isomorphism.
\end{proof}
\iffalse
\begin{rem}
Lemma~\ref{orbit-primitive} does not hold for pro-$p$ groups in general.
Let $F$ be the free pro-$p$ group on the infinite set $\{x_i \mid i \in \mathbb{N}\}$. Define $\varphi:F \to F$ by $\varphi(x_i)=x_{i+1}$ for all $i \in \mathbb{N}$.
Then $\varphi$ induces an injective endomorphism on $F / \Phi(F)$. Hence, $\varphi$ maps primitive elements to primitive elements, but it is certainly not surjective.
\end{rem}
\fi
\begin{definition}
Let $G$ be a profinite group. The automorphic orbit of an element $g \in G$ is defined by
\[\textrm{Orb}(g) = \{\alpha(g) \mid \alpha \in \textrm{Aut}(G) \}.\]
\end{definition}  

\begin{thm}
\label{orbit}
Let $G$ be a finitely generated pro-$p$ group, and suppose that $\textrm{Aut}(G)$ acts transitively on $G / \Phi(G)$. Let $1 \neq u \in G$.
If $\varphi$ is an endomorphism of $G$ such that $\varphi(\textrm{Orb}(u)) \subseteq \textrm{Orb}(u)$, then $\varphi$ is an automorphism. 
\end{thm}
\begin{proof}
Suppose that $\varphi$ is an endomorphism of $G$ satisfying  $\varphi(\textrm{Orb}(u)) \subseteq \textrm{Orb}(u)$. By Lemma~\ref{retract-intersection} $(c)$, there exists a retract $H$ of $G$ minimal among retracts of $G$ that contain $\langle u \rangle$. 
It follows from Lemma~\ref{retract-intersection} $(a)$ that $u$ is not contained in any proper retract of $H$. Being a retract of a finitely generated group, $H$ is finitely generated. 
Hence, by Theorem~\ref{test-retract}, $u$ is a test element of $H$. 

Let $r:G \to H$ be a retraction. Since $\varphi(\textrm{Orb}(u)) \subseteq \textrm{Orb}(u)$, there exists $\alpha \in \textrm{Aut}(G)$ such that $\varphi(u) = \alpha(u)$.  Consider the homomorphism $\beta : H\to H $ defined by $\beta: = {r \circ ({\alpha}^{-1}\circ \varphi)}_{| H}$. We have $$\beta(u) = r(\alpha^{-1}(\varphi(u))) = r(\alpha^{-1}(\alpha(u))) = r(u) = u.$$ Since $u$ is a test element of $H$, $\beta$ is an isomorphism. Let $x \in H$ be a primitive element of $H$. 
Since $H$ is a retract of $G$, $\Phi(H)=H \cap \Phi(G)$ (see the proof of Corollary~\ref{test-free}). Hence, $x \notin \Phi(G)$ and thus $x$ is a primitive element of $G$.
Furthermore, $\beta(x)$ is a primitive element of $H$ and $(\alpha^{-1}\circ \varphi)(x)$ is a primitive element of $G$. Indeed, $(\alpha^{-1}\circ \varphi)(x) \in \Phi(G)$ implies $\beta(x)=r((\alpha^{-1}\circ \varphi)(x)) \in \Phi(H)$, a contradiction with the primitivity of $\beta(x)$.
As $\alpha$ is an automorphism of $G$, it follows that $\varphi(x) = \alpha((\alpha^{-1}\circ \varphi)(x))$ is a primitive element of $G$.

Now let $y\in G$ be an arbitrary primitive element of $G$. Since $\textrm{Aut}(G)$ acts transitively on $G / \Phi(G)$, there is 
$\gamma \in \textrm{Aut}(G)$ such that $\gamma(x) y^{-1} \in \Phi(G)$. Observe that $\gamma(H)$ is a retract of $G$ and it is minimal among retracts of $G$ that contain $\langle \gamma(u) \rangle$. Moreover, $\gamma(x)$ is a primitive element of $\gamma(H)$ . 
Hence, we can repeat the argument from the previous paragraph with $u$ replaced by $\gamma(u)$ ($\textrm{Orb}(\gamma(u))=\textrm{Orb}(u)$) and conclude that $\varphi(\gamma(x))$ is a primitive element of $G$. 
From $\varphi(\gamma(x))\varphi(y)^{-1}=\varphi(\gamma(x)y^{-1}) \in \Phi(G)$, we get that $\varphi(y)$ is a primitive element of $G$. 
Hence, $\varphi$ maps primitive elements to primitive elements.  By Lemma~\ref{orbit-primitive}, $\varphi$ is an isomorphism.
\end{proof}

\begin{cor}
\label{orbit-free}
Let $F$ be a relatively free pro-$p$ group of finite rank, and let $1 \neq u \in F$.
If $\varphi$ is an endomorphism of $F$ such that $\varphi(\textrm{Orb}(u)) \subseteq \textrm{Orb}(u)$, then $\varphi$ is an automorphism. 
In particular, this statement is true for free pro-$p$ groups of finite rank.
\end{cor}
\begin{proof}
Since $\textrm{Aut}(F)$ acts transitively on $F/\Phi(F)$, Theorem~\ref{orbit} applies.
\end{proof}
\iffalse
\begin{rem}
We would like to note that one can use Proposition~\ref{varhi_infinity} to give an alternative proof of Theorem~\ref{orbit} that does not relay on Theorem~\ref{test-retract}.
Here is the gist of the argument. In the above proof, the minimal retract $H$ that contains $u$ can be replaced by $(\alpha^{-1} \circ \varphi)^{\infty}(G)$, and then one would use the fact that $\phi_{\infty}$  is an automorphism.
\end{rem}
\fi
We give an example of an endomorphism $\varphi$ of a finitely generated pro-$p$ group $G$ that is not an automorphism and for which there exists an element $1 \neq u \in G$ such that  $\varphi(\textrm{Orb}(u)) \subseteq \textrm{Orb}(u)$. 
Let $G = H \coprod C_p$ where $H=\langle x \rangle$ is an infinite procyclic pro-$p$ group and $C_p=\langle u \rangle$ is a cyclic group of order $p$.  
Define $\varphi: G \to G$ by $\varphi(x)=x^p$ and $\varphi(u)=u$. Let $\alpha \in \textrm{Aut}(G)$. Since all elements of finite order in $G$ belong to conjugates of $C_p$ (see \cite{Ribes2}), $\alpha(u)=wu^{k}w^{-1}$ for some $w \in G$ and 
$1 \leq k \leq p - 1$. 
Define $\beta_{k}:G \to G$ by $\beta_{k}(x)=x$ and $\beta_{k}(u)=u^{k}$. Clearly, $\beta_{k}$ is an automorphism of $G$ and
$\varphi(\alpha(u))=\varphi(w)u^{k}\varphi(w)^{-1}=(\rho_{\varphi(w)} \circ \beta_{k})(u)$ where $\rho_{\varphi(w)}$ denotes conjugation by $\varphi(w)$.

\section{Test elements of free pro-$p$ groups}
\label{test-free pro-p}
\subsection{$\mathcal{T}$-test arrangements}
\begin{pro}
\label{test-power}
Let $F$ be a free pro-$p$ group. If $u \in F$ is a test element of $F$, then so is $u^{\alpha}$ for every $0 \neq \alpha \in \mathbb{Z}_p$. 
\end{pro}
\begin{proof}
Suppose that $u$ is a test element of $F$, and let $0 \neq \alpha \in \mathbb{Z}_p$.
If $\varphi: F \to F$ is a homomorphism such that $\varphi(u^{\alpha}) = u^{\alpha}$, then by unique extraction of roots $\varphi(u) = u$. 
Hence, $\varphi$ is an automorphism and $u^{\alpha}$ is a test element of $F$. 
\end{proof}

We write $F(a_1, \ldots, a_n)$ for the free pro-$p$ group on the finite set $\{a_1, \ldots, a_n\}$; if there is no possibility of misunderstanding, we will also write $F$  for $F(a_1, \ldots, a_n)$.  
Let $G$ be a pro-$p$ group, $g_1, \ldots, g_n \in G$, and $w=w(x_1, \ldots, x_n)$ be an element in the free pro-$p$ group $F(x_1, \ldots, x_n)$. Then $w(g_1, \ldots, g_n) \in G$ denotes  the image of $w$
under the homomorphism $\varphi:F \to G$ defined by $\varphi(x_i)=g_i$ for all $1 \leq i \leq n$.

\begin{pro}
\label{coprod-test}
Let $F'=A_1 \coprod A_2 \coprod \ldots \coprod A_n$  be a free pro-$p$ group of finite rank.
For each $i =1, \ldots , n$, let $u_i \in A_i$ be a test element of $A_i$. 
 If $ w(x_1, \ldots , x_n)$ is a test element of the free pro-$p$ group $F(x_1, \ldots, x_n)$, then
$w(u_1, \ldots , u_n)$ is a test element of $F'$.
\end{pro}
\begin{proof}
Let $ w(x_1, \ldots , x_n)$ be a test element of the free pro-$p$ group $F(x_1, \ldots, x_n)$.
Clearly, $K= \langle u_1,  \ldots , u_n \rangle$ is a free pro-$p$ group with basis  $\{u_1, \dots, u_n\}$ and $t=w(u_1, \ldots, u_n) \in K$ is a test element of $K$.
Let $C$ be the algebraic closure of $\langle t \rangle$ in $F'$. By Proposition~\ref{freefactors-intersection}, $K \cap C$ is a free factor of $K$.
Since $t \in  K \cap C$ is a test element of $K$, we have $ K \cap C=K$. Consequently, $u_i \in C$ for all $1 \leq i \leq n$. Moreover, for each $1 \leq i \leq n$, $A_i \cap C \leq_{\textrm{ff}} A_i$, and 
since $u_i \in A_i \cap C$ is a test element of $A_i$, it follows that $A_i \leq C$. Hence, $C=F'$ and by Corollary~\ref{test-free}, $t$ is a test element of $F'$. 
\end{proof}

\begin{rem}
Note that  in the proof of Proposition~\ref{coprod-test} we only used the fact that $F'$ is generated by the subgroups $A_1, \ldots, A_n$ and that $u_1, \ldots, u_n$ form a basis of the subgroup they generate.   
\end{rem}

For each $n \geq 1$, let $\mathcal{T}_n$ be a set (possibly empty) of test elements of the free pro-$p$ group $F(x_1, \ldots, x_n)$;
set $\mathcal{T}=\bigcup_{n \geq 1}\mathcal{T}_n$. We define the notion of a \emph{$\mathcal{T}$- test arrangement} recursively:
\begin{enumerate}[i)]
\item every element in $\mathcal{T}_n$ is a $\mathcal{T}$- test arrangement of weight $n$;
\item given $w \in \mathcal{T}_n$ and $\mathcal{T}$-test arrangements $u_i$ of respective weights $k_i, 1 \leq i \leq n$, then 
$w(u_1, \ldots, u_n)$ is a $\mathcal{T}$-test arrangement of weight $k_1 +  \ldots + k_n$;
\item every $\mathcal{T}$-test arrangement is obtained from $i)$ and $ii)$ in a finite number of steps.
\end{enumerate}

In order to avoid any confusion, we emphasize that a $\mathcal{T}$-test arrangement is nothing but a special type of word in the following symbols: elements in  $\mathcal{T}$, comma, and the parenthesis `(' and `)'. 

Let $G$ be a pro-$p$ group, and let $w$ be a $\mathcal{T}$-test arrangement of weight $n$. Given elements $g_1, \ldots, g_n$ in $G$, we define 
$w(g_1, \ldots, g_n)$, recursively as follows:
\begin{enumerate}[i)]
\item if $w \in \mathcal{T}_n$, then $w(g_1, \ldots, g_n)$ was already defined;
\item if $w=v(u_1, \ldots, u_m)$ for some $v \in \mathcal{T}_m$ and $\mathcal{T}$-test arrangements $u_i$ of respective weights $k_i, 1 \leq i \leq m$, 
with $k_1 + \ldots + k_m=n$, then 
\[w(g_1, \ldots, g_n)=v(u_1(g_1, \ldots, g_{k_1}), \ldots, u_m(g_{k_1+\ldots+ k_{m-1}+1}, \ldots, g_n)).\]

\end{enumerate}

\begin{pro}
\label{test-arrangement}
Let $F(a_1, \ldots, a_n)$ be a free pro-$p$ group, and let $w$ be a $\mathcal{T}$-test arrangement of weight $n$.
Then $w(a_1, \ldots, a_n)$ is a test element of $F$.
\end{pro}
\begin{proof}
If $w \in \mathcal{T}_n$, then $w(a_1, \ldots, a_n)$ is clearly a test element of $F$.
Assume that $w=v(u_1, \ldots, u_m)$ for some $v \in \mathcal{T}_m$ and $\mathcal{T}$-test arrangements $u_i$ of respective weights $k_i, 1 \leq i \leq m$.
For $i=1, \ldots, m$, let  $A_i=\langle a_{k_1 + \ldots + k_{i-1}+1}, \ldots, a_{k_1+ \ldots + k_i} \rangle$.
Clearly, $F =A_1 \coprod \ldots \coprod A_m$. By induction, $u_i(a_{k_1 + \ldots + k_{i-1}+1}, \ldots, a_{k_1+ \ldots + k_i})$ is a test element of $A_i$ for all $1 \leq i \leq m$.
It follows from Proposition~\ref{coprod-test} that $w(a_1, \ldots, a_n)$ is a test element of $F$.
\end{proof}

For each $n \geq 1$, let $\mathbb{T}_n$ be the set of all test elements of $F(x_1, \ldots, x_n)$, and set $\mathbb{T} = \bigcup_{n \geq 1} \mathbb{T}_n$. 
We employ Proposition~\ref{test-arrangement} to turn $\mathbb{T}$ into a (graded) universal algebra with signature $\mathbb{T}$.
The $n$-ary operation $\alpha_{w}:\mathbb{T}^n \to \mathbb{T}$ corresponding to $w \in \mathbb{T}_n$ is defined as follows: let $u_i \in \mathbb{T}_{k_i}$ for $1 \leq i \leq n$, and set
$m=k_1 + \ldots + k_n$; then $\alpha_w(u_1, \ldots, u_n)$ is the test element $w(u_1, \ldots, u_n)(x_1, \ldots, x_m) \in \mathbb{T}_m$ of $F(x_1, \ldots, x_m)$.
However,  in order to turn this observation into a practical tool for discovering test elements of free pro-$p$ groups, it is helpful to consider $\mathbb{T}$ as a universal algebra with a smaller signature $\mathcal{T} \subset \mathbb{T}$ 
consisting of test elements that are already known. Then new test elements are obtained from the subalgebra generated by $\mathcal{T}$:
\[\mathcal{T}'=\langle \mathcal{T} \rangle_{\mathcal{T}} = \{ w(x_1, \ldots, x_n) \mid n \geq 1, w \text{ is a } \mathcal{T} \text{-test arrangement of weight } n \}.\]  
Any hope that one may use this method recursively is shattered by the simple observation that $\langle \mathcal{T}' \rangle_{\mathcal{T}'}=\mathcal{T}'$. 
               
The above discussion prompts the following definition.

\begin{definition}
Let $w \in F(x_1, \ldots, x_n)$ be a test element of F. The rank of $w$ is defined by
\[rank(w)=min\{m \in \mathbb{N} \mid w \in \langle \mathcal{T} \rangle_{\mathcal{T}} \text{ where } \mathcal{T}=\mathbb{T}_1 \cup \ldots \cup \mathbb{T}_m \}\] 
\end{definition}

Observe that every nontrivial element in the free pro-$p$ group of rank one (i.e. $\mathbb{Z}_p$) is a test element. (These are the only test elements of rank one.)  
However, all test elements of a free pro-$p$ group of rank greater than one are contained in the Frattini subgroup.
Let $F$ be a free pro-$p$ group of rank two. It follows from Corollary~\ref{test-free} that $w \in F$ is a test element of $F$ if and only if it is not a power (with exponent in $\mathbb{Z}_p$) of a primitive element.
An immediate consequence is that every nontrivial element in the commutator subgroup of $F$ is a test element. Indeed, the image under the quotient homomorphism onto the abelianization 
$F / [F,F]$ of a nonzero power of a primitive element is nontrivial. 
    
Next we give an alternative description of test elements of the free pro-$p$ group of rank two; to that end we introduce some notation. 
Let $F(x_1, \ldots, x_n)$ be a free pro-$p$ group. For each $i = 1, \ldots, n$, we define a homomorphism $\sigma_{x_i}:F \to \mathbb{Z}_p$ by
$\sigma_{x_i}(x_j)=\delta_{ij}$ ($\delta_{ij}$ is the Kronecker delta). If we think of an element $w \in F$ as a word (finite or infinite), then
$\sigma_{x_i}(w)$ can be described as the sum of the exponents of all occurrences of $x_i$ in $w$.  
\begin{pro}
\label{rank2}
Let $F(x_1,x_2)$ be a free pro-$p$ group of rank two, and let $1 \neq w\in F$.
Suppose that $w$ is not a $p$-th power of another element in $F$. Then
$w$ is a test element of $F$ if and only if $\sigma_{x_1}(w), \sigma_{x_2}(w) \in p\mathbb{Z}_p$.
\end{pro}
\begin{proof}
Assume that $\alpha=\sigma_{x_1}(w) \notin p\mathbb{Z}_p$. Then $\alpha$ is a unit in $\mathbb{Z}_p$.
Define $\varphi:F \to F$ by $\varphi(x_1)=w^{\alpha^{-1}}$ and $\varphi(x_2)=1$. Clearly, $\varphi(F)=\langle w \rangle \neq F$ and 
$\varphi(w)=(w^{\alpha^{-1}})^{\alpha}=w$. Hence, $w$ is not a test element of $F$.

Now suppose that $w$ is not a test element of $F$. It follows from Corollary~\ref{test-free} that $w$ is contained in a free factor of $F$, necessarily of rank one. 
Therefore, $w=y^{\beta}$ for some primitive element $y \in F$ and some $\beta \in \mathbb{Z}_p$. Since $w$ is not a $p$-th power, we have $\beta \in \mathbb{Z}_p \setminus p\mathbb{Z}_p$.
Moreover, $y \notin \Phi(F)$ implies that either $\sigma_{x_1}(y) \notin p\mathbb{Z}_p$ or $\sigma_{x_2}(y) \notin p\mathbb{Z}_p$. Consequently, either
$\sigma_{x_1}(w)=\sigma_{x_1}(y^{\beta})=\beta \sigma_{x_1}(y) \notin p\mathbb{Z}_p$ or $\sigma_{x_2}(w)=\beta \sigma_{x_2}(y) \notin  p\mathbb{Z}_p$.
\end{proof}

To illustrate the importance of  Proposition~\ref{rank2}, we will prove that $t=x_1^{\alpha_1}x_2^{\alpha_2}$, $ \alpha_1, \alpha_2 \in p \mathbb{Z}_p \setminus \{0\},$  is a test element of $F(x_1, x_2)$ (see also Proposition~\ref{almost_primitive - examples} $(d)$).
For $i=1,2$, $\alpha_i=p^{n_i}\beta_i$ for some $n_i \in \mathbb{N}$ and $\beta_i \in \mathbb{Z}_p \setminus p\mathbb{Z}_p$. Note that $\{x_1^{\beta_1}, x_2^{\beta_2}\}$ is a basis of $F$.
Therefore, in order to show that $t$ is not a $p$-th power in $F$, it suffices to observe that $t$ is not a $p$-th power in the free discrete group generated by  $\{x_1^{\beta_1}, x_2^{\beta_2}\}$ (see \cite[Proposition~2]{Baumslag}). Since  $\sigma_{x_i}(t)=\alpha_i \in p \mathbb{Z}_p$ for $i=1,2$, it follows from Proposition~\ref{rank2} that $t$ is a test element of $F$.

 Now set $\mathcal{T}_2=\{x_1^{\alpha_1}x_2^{\alpha_2} \mid \alpha_1, \alpha_2 \in p \mathbb{Z}_p \setminus \{0\}\}$ 
and $\mathcal{T}_n = \emptyset$ for $n \neq 2$. Then $\langle \mathcal{T} \rangle_{\mathcal{T}}$ contains many new test elements of free pro-$p$ groups.
For example, for all $\alpha_i, \beta_i, \gamma_i \in p \mathbb{Z}_p \setminus \{0\}$, $1 \leq i \leq 2$, $(x_1^{\alpha_1}x_2^{\alpha_2})^{\beta_1}x_3^{\beta_2}$ and $x_1^{\beta_1}(x_2^{\alpha_1}x_3^{\alpha_2})^{\beta_2}$  are test element of $F(x_1, x_2, x_3)$, and 
$(x_1^{\alpha_1}x_2^{\alpha_2})^{\gamma_1}(x_3^{\beta_1}x_4^{\beta_2})^{\gamma_2}$ is a test element of $F(x_1, x_2, x_3, x_4)$.

For ease of reference, in the following two propositions we collect some immediate consequences of the results of this subsection. 
   
\begin{pro}
\label{test_powers}
Let $w$ be a test element of the free pro-$p$ group $F(x_1, \ldots, x_n)$.
For all $\alpha_1, ..., \alpha_n \in \mathbb{Z}_p \setminus {\{0 \}}$, $w(x_1^{\alpha_1}, ..., x_n^{\alpha_n})$ is also a test element of $F$.
\end{pro}
\begin{proof}
Set $\mathcal{T}_1=\{x_1^{\alpha_1}, \ldots ,  x_1^{\alpha_n}\}$, $\mathcal{T}_n =\{w\}$, and $\mathcal{T}_k=\emptyset$ for $k \notin \{1,n\}$.
Then $w(x_1^{\alpha_1}, \ldots , x_1^{\alpha_n})$ is a $\mathcal{T}$-test arrangement, and by Proposition~\ref{test-arrangement}, $w(x_1^{\alpha_1}, \ldots, x_n^{\alpha_n})=w(x_1^{\alpha_1}, \ldots , x_1^{\alpha_n})(x_1, \ldots, x_n)$ is a test element of $F$.
\end{proof}

\begin{pro}
\label{higher commutator}
Let $t = [x_1, ... , x_n]$ be any higher commutator in $F(x_1, \ldots, x_n)$ (that is, $t$ is a commutator of weight $n$ involving all $n$ letters $x_1, ..., x_n$ with arbitrary disposition of commutator brackets). Then $t$ is a test element of $F$. 
\end{pro}
\begin{proof}
Set $\mathcal{T}_1=\{x_1\}$, $\mathcal{T}_2=\{[x_1,x_2]\}$ and $\mathcal{T}_k=\emptyset$ for $k \notin \{1, 2\}$. It is easy to see that there is a $\mathcal{T}$-test arrangement $w$  of weight $n$ satisfying 
$w(x_1, \ldots, x_n)=t$. Hence, $t$ is a test element of $F$ by Proposition~\ref{test-arrangement}.    
\end{proof}
Of course, in the proof of the previous proposition we could have put $\mathcal{T}_2=[F(x_1,x_2), F(x_1,x_2)]$ and thus obtain a much larger class of test elements.
Anyhow, if the reader finds our description of test elements of $F(x_1, x_2)$ satisfying, then they would agree that through the notion of a $\mathcal{T}$-test arrangement we have described all test elements of rank two in free pro-$p$ groups.
In order to give examples of test elements of higher rank we introduce the concept of an almost primitive element.

\subsection{Almost primitive elements}
\begin{definition}
Let $G$ be a pro-$p$ group. A non-primitive element $g \in G$
is termed almost primitive element of $G$ 
if it is a primitive element of every  proper subgroup of $G$ that contains it.
\end{definition}

If $g \in G$ is an almost primitive element of $G$, then $g \in \Phi(G)$ and $g \notin \Phi(H)$ whenever $H \lneq  G$.
Since every proper subgroup of $G$ is contained in a maximal subgroup, it follows that $g \in G$ is an almost primitive elements of $G$ if and only if 
$$g \in \Phi(G) \setminus \bigcup_{M} \Phi(M)$$
where $M$ runs through all maximal subgroups of $G$. Observe that if $G$ is finitely generated, then the almost primitive elements of $G$ form a clopen set.

\begin{pro}
\label{almost primitive is test}
Let $G$ be a finitely generated pro-p group and $H \leq G$.
\begin{enumerate}[(a)]
\item If $\Phi(H)$ contains an almost primitive element of $G$, then $H=G$.
\item Every almost primitive element of $G$ is a test element of $G$.
\end{enumerate}  
\end{pro}
\begin{proof}
Part (a) is clear. Let $g \in G$ be an almost primitive element of $G$.
Let $\varphi: G \to G$ be a homomorphism that fixes $g$. Then $\varphi(\Phi(G)) = \Phi(\varphi(G))$.
Hence, $g \in \Phi(\varphi(G))$ and $\varphi(G)=G$ by part $(a)$. It follows from the Hopfian property of finitely generated pro-$p$ groups that $\varphi$ is an isomorphism, and thus $g$ is a test element of $G$. 
\end{proof}

\begin{pro}
Let $a$ be an almost primitive element of the free pro-$p$ group $F(x_1, \ldots, x_n)$.
Then $rank(a)=n$.
\end{pro}
\begin{proof}
If $n=1$, then there is nothing to prove. Suppose that $n > 1$, and 
assume to the contrary that $a \in \langle \mathcal{T} \rangle_{\mathcal{T}}$ where 
$\mathcal{T} = \mathbb{T}_1 \cup \ldots \cup \mathbb{T}_m$ for some $m<n$.
Then there is a $\mathcal{T}$-test arrangement $w$ of weight $n$ such that $w(x_1, \ldots, x_n)=a$.
Now $w=v(u_1, \ldots, u_s)$ for some $v \in \mathbb{T}_s, s \leq m < n,$ and $\mathcal{T}$-test arrangements $u_i$ of respective weights $k_i, 1 \leq i \leq s$. 
If $s=1$, then $v=x_1^{\alpha} \in \mathbb{T}_1, \alpha \neq 0,$ and $w(x_1, \ldots, x_n)=u_1(x_1, \ldots, x_n)^{\alpha}$.
Since $u_1(x_1, \ldots, x_n)$ is a test element of $F(x_1, \ldots, x_n)$ and $n>1$, we have $u_1(x_1, \ldots, x_n) \in \Phi(F)$.
If $\alpha \in p\mathbb{Z}_p$, then $a=w(x_1, \ldots, x_n) \in \Phi^2(F)$, which contradicts the fact that $a$ is an almost primitive element of $F$. Hence, $\alpha \in \mathbb{Z}_p \setminus p\mathbb{Z}_p$ and 
$u_1(x_1, \ldots, x_n)=a^{\alpha^{-1}}$. Clearly, $a^{\alpha^{-1}}$ is an almost primitive element of $F$ and 
$rank(a^{\alpha^{-1}})=rank(a)=m$. This shows that we can assume that $s>1$.
  
Since $k_1 + \ldots + k_s=n$, it follows that $k_j > 1$ for some $1 \leq j \leq s$.
Let $M \leq F$ be a maximal subgroup of $F$ satisfying $x_i \in M$ for all $i \notin \{k_1 + \ldots + k_{j-1}+1, \ldots, k_1 + \ldots + k_j\}$. Observe that $u_j(x_{k_1 + \ldots + k_{j-1}+1}, \ldots, x_{k_1 + \ldots + k_j}) \in \Phi(F) \leq M$.
Therefore, 
\[a=w(x_1, \ldots, x_n)=v(u_1(x_1, \ldots, x_{k_1}), \ldots, u_m(x_{k_1+\ldots+ k_{m-1}+1}, \ldots, x_n)) \in \Phi(M),\]
which is a contradiction since $a$ is an almost primitive element of $F$.
\end{proof}

\begin{pro}
\label{almost primitive - free factor}
Let $F=A_1 \coprod \ldots \coprod A_n$ be a free pro-$p$ group, and let $u_i \in A_i, 1 \leq i \leq n$.
Then $w=u_1 \ldots u_n$ is an almost primitive element of $F$ if and only if $u_i$ is an almost primitive element of $A_i$ for 
all $1 \leq i \leq n$. 
\end{pro}
\begin{proof}
The homomorphism $\alpha:A_1 / \Phi(A_1) \times  \ldots \times A_n / \Phi(A_n) \to F/\Phi(F)$ defined by $\alpha(x_1\Phi(A_1), \ldots, x_n\Phi(A_n))=x_1\ldots x_n\Phi(F)$ is an isomorphism.
It follows that $w \in \Phi(F)$ if and only if $u_i \in \Phi(A_i)$ for all $1 \leq i \leq n$.

Suppose that $u_i$ is an almost primitive element of $A_i$ for all $1 \leq i \leq n$. Let $M$ be a maximal subgroup of $F$. By the Kurosh subgroup theorem, 
\[M = A_1 \cap M \coprod \ldots \coprod A_n \cap M \coprod C\]
for some $C \leq M$. Moreover, there is $j, 1 \leq j \leq n,$ for which $|M:A_j \cap M|=p$. (Otherwise $A_i \leq M$ for all $i$, and  $F \leq M$, contradicting the fact that $M$ is a proper subgroup of $F$.)
By assumption $u_j \in A_j \cap M$ is an almost primitive element of $A_j$, and thus $u_j \notin \Phi(A_j \cap M)$. It follows from the isomorphism (defined as above)  
$$(A_1 \cap M) / \Phi(A_1 \cap M) \times  \ldots \times (A_n \cap M) / \Phi(A_n \cap M) \times C / \Phi(C) \cong M/\Phi(M),$$
that $w \notin \Phi(M)$. Therefore, $w$ is an almost primitive element of $F$.

Now assume that $w$ is an almost primitive element of $F$. Given $j, 1 \leq j \leq n$, and a maximal subgroup $N$ of $A_j$, we need to prove that $u_j \notin \Phi(N)$.
Let $M$ be a maximal subgroup of $F$ satisfying $A_j \cap M=N$ and $A_i \leq M$ for all $i \neq j$. Then $u_i \in \Phi(A_i) \leq \Phi(M)$ for all $i \neq j$, and 
$\Phi(M) \cap N = \Phi(N)$. Therefore, $w=u_1 \ldots u_n \notin \Phi(M)$ implies $u_j \notin \Phi(N)$. 
\end{proof}

\begin{pro}
\label{almost_primitive - examples}
\begin{enumerate}[a)]
\item For every $\alpha \in p\mathbb{Z}_p \setminus p^2\mathbb{Z}_p$, $x_1^{\alpha}$ is an almost primitive element of the procyclic group $F(x_1)$.
\item The commutator $[x_1, x_2]$ is an almost primitive element of $F(x_1, x_2)$.
\item Given $n,k \geq 0$ and $\alpha_i \in p\mathbb{Z}_p \setminus p^2\mathbb{Z}_p, 1 \leq i \leq n$, then 
\[x_1^{\alpha_1} \ldots x_n^{\alpha_n}[x_{n+1}, x_{n+2}] \ldots [x_{n+2k-1}, x_{n+2k}]\]
 is an almost primitive element of $F(x_1, \ldots, x_n, x_{n+1}, \ldots, x_{n+2k})$.
\item Given $n,k \geq 0$, $\alpha_i \in p\mathbb{Z}_p \setminus \{0\}, 1 \leq i \leq n$, and 
$\beta_i \in \mathbb{Z}_p \setminus \{0\}, 1 \leq i \leq 2k$, then
\[x_1^{\alpha_1} \ldots x_n^{\alpha_n}[x_{n+1}^{\beta_1}, x_{n+2}^{\beta_2}] \ldots 
[x_{n+2k-1}^{\beta_{2k-1}}, x_{n+2k}^{\beta_{2k}}]\]
 is a test element of $F(x_1, \ldots, x_n, x_{n+1}, \ldots, x_{n+2k})$.
\end{enumerate}
\end{pro}
\begin{proof}
Part $(a)$ is obvious.
\\
$(b)$ Clearly, $[x_1, x_2] \in \Phi(F(x_1, x_2))$. Let $M$ be a maximal subgroup of $F$, and assume to the contrary that
$[x_1, x_2] \in \Phi(M)$. Since $\Phi(M)$ is a characteristic subgroup of $M$ and $M \unlhd F$, it follows that
$\Phi(M) \unlhd F$; hence, $[x_1, x_2] \in \Phi(M)$ implies  $[F, F] \leq \Phi(M)$.
Moreover, $M / [F,F]$ is a maximal subgroup of the abelianization  $F_{\textrm{ab}}=F / [F,F] \cong \mathbb{Z}_p \times \mathbb{Z}_p$ , and 
$$d(M / [F,F])=\textrm{dim}_{\mathbb{F}_p} (M / [F,F]) / \Phi(M/ [F,F])=\textrm{dim}_{\mathbb{F}_p} M / \Phi(M)=p+1 \geq 3,$$ 
which is a contradiction since every subgroup of $\mathbb{Z}_p \times \mathbb{Z}_p$ can be generated by two elements.
\\
$(c)$ Let $A_i = \langle x_i \rangle$ for $1 \leq i \leq n$, and $B_i=\langle x_{n+2i-1}, x_{n + 2i} \rangle$ for $1 \leq i \leq k$.
Then $F(x_1, \ldots, x_n, x_{n+1}, \ldots, x_{n+2k})=A_1 \coprod \ldots \coprod A_n \coprod B_1 \coprod \ldots \coprod B_k$.
It follows from $(a)$ that for all $1 \leq i \leq n$,  $x_i^{\alpha_i} \in A_i$ is an almost primitive element of $A_i$. By part $(b)$, $[x_{n+ 2i-1}, x_{n+2i}] \in B_i$ is an almost primitive element of $B_i$ for all $1 \leq i \leq k$.
Now $(c)$ follows from Proposition~\ref{almost primitive - free factor}.
\\
$(d)$ This is an immediate consequence of $(c)$ and Proposition~\ref{test_powers}. 
\end{proof}

\begin{rem}
\label{almost primitive remark}
With a slight extension of the proof of Proposition~\ref{almost_primitive - examples} $(b)$ we can show that $x_1^{\alpha}[x_1,x_2]$ is an almost primitive element of $F(x_1, x_2)$ for all $\alpha \in p\mathbb{Z}_p$.
It is also true, but more difficult to prove, that every element of the form $x_1^{\alpha_1}x_2^{\alpha_2}[x_1, x_2]$ with $\alpha_1, \alpha_2 \in p\mathbb{Z}_p$ is an almost primitive element of $F(x_1, x_2)$ 
(the only difficult case is when $\alpha_1=\alpha_2=p$ ). These results will appear in a subsequent paper by the authors (\cite{SnoTan}) where we also introduce the notion of a generic element, which is a generalization of an almost primitive element. 
Here we would like to point out that using these results we get many other almost primitive elements. For example, it follows from Proposition~\ref{almost primitive - free factor} and Proposition~\ref{almost primitive - modify} $(b)$ that 
$x_1^{\alpha_1} \ldots x_{2n}^{\alpha_{2n}}[x_1,x_2] \ldots [x_{2n-1}, x_{2n}]$ is an almost primitive element of $F(x_1, \ldots, x_{2n})$.
\end{rem}

In the following proposition we describe two simple ways in which new almost primitive elements can be obtained. 

\begin{pro} 
\label{almost primitive - modify}
Let $G$ be a pro-$p$ group, and let $g \in G$ be an almost primitive element of $G$. Then
\begin{enumerate}[(a)]
\item For every $h \in \bigcap_{M} \Phi(M)$ where the intersection is over all maximal subgroups of $G$, $gh$ is an almost primitive element of $G$.
\item If $g=h_1 \ldots h_n$ for some $h_i \in \Phi(G), 1 \leq i \leq n$,  then for every permutation $\sigma$ of $\{1, \ldots, n\}$, $g_{\sigma}=h_{\sigma (1)} \ldots h_{\sigma (n)}$ is an almost primitive element of $G$.
\end{enumerate}
\end{pro}
\begin{proof}
$(a)$ Clearly, $gh \in \Phi(G)$. Let $M$ be a maximal subgroup of $G$. Then $gh \equiv g \textrm{ mod } \Phi(M)$, and thus $gh \notin \Phi(M)$.
\\
$(b)$ Let $g=h_1 \ldots h_n$ where $h_i \in \Phi(G)$ for $1 \leq i \leq n$.
Then $g_{\sigma} \in \Phi(G)$. Let $M \leq G$ be a maximal subgroup. Then $h_i \in M$ for all $1 \leq i \leq n$, and $g_{\sigma}\equiv g \textrm{ mod } \Phi(M)$.
Hence,  $g_{\sigma} \notin \Phi(M)$ and $g_{\sigma}$ is an almost primitive element of $G$.
\end{proof}

\begin{lem}
\label{going up}
Let $F$ be a free pro-$p$ group of finite rank, and let $H$ be a subgroup of $F$ (not necessarily finitely generated) with algebraic closure $F$.
If $w \in H$ is an almost primitive (resp. test) element of $H$, then it is a test element of $F$. In particular, every test element of a subgroup of finite index in $F$ is a test element of $F$.
\end{lem}
\begin{proof}
It is easy to see that if $w \in H$ is an almost primitive (resp. test) element of $H$, then the algebraic closure of $\langle w \rangle$ in $H$ is exactly $H$. 
Now let $A$ be the algebraic closure of  $\langle w \rangle$ in $F$.
By Proposition~\ref{freefactors-intersection}, $A \cap H$ is a free factor of $H$. Since $w \in A \cap H$, it follows that $H \leq A$.
But the algebraic closure of $H$ in $F$ is exactly $F$, hence $F=A$. By Corollary~\ref{test-free}, $w$ is a test element of $F$.
Now the last statement in the proposition follows from the obvious fact that the algebraic closure of a finite index subgroup of $F$ is exactly $F$. 
\end{proof}

\begin{lem}
\label{a lot of test elements}
Let $G$ be a pro-$p$ group and $w \in \Phi(G)$. Then the following holds:
\begin{enumerate}[(a)]
\item there exists $H \leq G$  for which $w \in H$ and $w$ is an almost primitive element of $H$;
\item if $G$ is a free pro-$p$ group with basis $\{a_1, \ldots, a_n\}$, then there exists a subset $\{a_{i_1}, \ldots, a_{i_m}\} \subseteq \{a_1, \ldots, a_n\}$ such that
for every test element $u$ of $F(x_1, \ldots, x_m)$ contained in $\Phi(F)$, $w \cdot u(a_{i_1}, \ldots, a_{i_m})$ is a test element of $G$.
\end{enumerate}
\end{lem}
\begin{proof}
$(a)$ Let $Sub(G)$ be the set of closed subgroups of $G$, and let
$$ \mathcal{S}=\{\mathcal{A} \subseteq Sub(G) \mid \mathcal{A} \text{ is filtered from below and } w \in \Phi(K) \text{ for all } K \in \mathcal{A}\}.$$
Note that $\mathcal{S} \neq \emptyset$ since $\{G\} \in \mathcal{S}$.
Order $\mathcal{S}$ by inclusion and let $\mathcal{C} \subseteq \mathcal{S}$ be a chain. It is easy to see that $\bigcup \mathcal{C} \in \mathcal{S}$.
By Zorn's lemma, there is a maximal element $\mathcal{A} \in \mathcal{S}$;  set $H = \bigcap \mathcal{A}$. Since $\mathcal{A}$ is filtered from below,
$\Phi(H) = \bigcap_{K \in \mathcal{A}}\Phi(K)$ and  thus $w \in \Phi(H)$. By maximality of $\mathcal{A}$, $w \notin \Phi(T)$ whenever $T \lneq H$. 
Hence, $w$ is an almost primitive element of $H$.
\\
$(b)$ By part $(a)$, there exists $H \leq G$ such that $w \in H$ and $w$ is an almost primitive element of $H$. 
Let $A$ be the algebraic closure of $H$ in $G$.  Then $G=A \coprod B$ where we can choose $B=\langle a_{i_1}, \ldots, a_{i_m} \rangle $ for some  
$\{a_{i_1}, \ldots, a_{i_m}\} \subseteq \{a_1, \ldots, a_n\}$. 
We may assume that $B \neq \{1\}$, otherwise it follows from Lemma~\ref{going up} that $w$ is a test element of $G$.
Let $u \in\Phi( F(x_1, \ldots, x_m))$ be a test element of $F$; then $u(a_{i_1}, \ldots, a_{i_m})$ is a test element of  $B$. By part $(a)$, there is $K \leq B$ such that $u(a_{i_1}, \ldots, a_{i_m}) \in K$ and $u(a_{i_1}, \ldots, a_{i_m})$ is an almost primitive element of $K$.
Observe that $T=\langle H \cup K \rangle=H \coprod K$ (see \cite[Corollary 9.1.7]{Ribes1}).
By Proposition~\ref{almost primitive - free factor}, $t=w \cdot u(a_{i_1}, \ldots, a_{i_m}) \in T$ is an almost primitive element of $T$.
If $C$ is the algebraic closure of $T$ in G, then by Proposition~\ref{freefactors-intersection}, $H \leq A \cap C \leq_{\textrm{ff}} A$ and $\langle u \rangle \leq B \cap C \leq_{\textrm{ff}} B$;
thus $A, B \leq C$ and $C=G$. It follows from Lemma~\ref{going up} that $t$ is a test element of $G$.

\end{proof}
\begin{thm}
\label{Frattini dense}
Let $F(x_1, \ldots, x_n)$ be a non-abelian free pro-$p$ group of finite rank. Then the set of test elements of $F$ is a dense subset of $\Phi(F)$. 
\end{thm}
\begin{proof}
Clearly, all test elements of $F$ are contained in $\Phi(F)$.
Given $w \in \Phi(F)$ and $N \leq_{\textrm{o}}F$, we need to show that $wN$ contains a test element of $F$. 
By Lemma~\ref{a lot of test elements} $(b)$, there exists a subset $\{x_{i_1}, \ldots, x_{i_m}\} \subseteq \{x_1, \ldots, x_n\}$  such that
for every test element $u(x_{i_1}, \ldots, x_{i_m})$ of $H=\langle x_{i_1}, \ldots, x_{i_m} \rangle$ contained in $\Phi(H)$, $w \cdot u(x_{i_1}, \ldots, x_{i_m})$ is a test element of $F$.
Since $|F:N| < \infty$, there exists a natural number $s \geq 0$ such that $x_{i_k}^{p^s} \in N$ for all $1 \leq k \leq m$.
By Proposition~\ref{test_powers}, if $u(x_{i_1}, \ldots, x_{i_m})$ is a test element of $H$, then so is 
$u(x_{i_1}^{p^s}, \ldots, x_{i_m}^{p^s})$.
Therefore, for every test element $u$ of $H$ contained in $\Phi(H)$, $w \cdot u(x_{i_1}^{p^s}, \ldots, x_{i_m}^{p^s}) \in wN$ is a test element of $F$.
\end{proof}

\section{Test elements of Demushkin groups}
\label{demushkin}

%\begin{definition} A pro-$p$ group $G$ is called a $Demushkin$ group if it satisfies the following conditions:
%\begin{itemize}
%\item[(i)] $\textrm{dim} _{\mathbb{F}_p} H^1(G,\mathbb{F}_p) < \infty$,
%\item[(ii)] $\textrm{dim} _{\mathbb{F}_p} H^2(G,\mathbb{F}_p)=1$, \textrm{ and }
%\item[(iii)] the cup-product $H^1(G,\mathbb{F}_p)\times H^1(G,\mathbb{F}_p) \to H^2(G,\mathbb{F}_p) \cong \mathbb{F}_p$ is a non-degenerate bilinear form.
%\end{itemize}
%\end{definition}

\begin{definition} Let $n\geq 1$ be an integer. A pro-$p$ group $G$ is called a \emph{Poincar\'{e} group of dimension $n$}  if it satisfies the following conditions:
\begin{itemize}
\item[(i)] $\textrm{dim} _{\mathbb{F}_p} H^i(G,\mathbb{F}_p) < \infty$  for all $i$,
\item[(ii)] $\textrm{dim} _{\mathbb{F}_p} H^n(G,\mathbb{F}_p)=1$, \textrm{ and }
\item[(iii)] the cup-product $H^i(G,\mathbb{F}_p)\times H^{n-i}(G,\mathbb{F}_p) \to H^n(G,\mathbb{F}_p) \cong \mathbb{F}_p$ is a non-degenerate bilinear form for all $0 \leq i \leq n$.
\end{itemize}
\end{definition}
A Poincar\'{e}  group of dimension 2 is called a \emph{Demushkin group} (see \cite{Serre1}).
 Demushkin groups play an important role in algebraic number theory. For instance, if $k$ is a $p$-adic number field containing a primitive $p$-th root of unity and $k(p)$ is the maximal $p$-extension of $k$, then $\textrm{Gal}(k(p)/k)$ is a Demushkin group (see \cite{Wingberg}). Demushkin groups also appear in topology since pro-$p$ completions of orientable surface groups are Demushkin groups.

Let $G$ be a Demushkin group with $d=d(G)$. Since $\textrm{dim} _{\mathbb{F}_p} H^2(G,\mathbb{F}_p)=1$, i.e.,  $G$ is a one related pro-$p$ group, there is an epimorphism $\pi: F \twoheadrightarrow G$ where $F$ is a free pro-$p$ group of rank $d$ and $\textrm{ker}(\pi)$ is generated as a closed normal subgroup by one element $r \in \Phi(F)=F^p[F, F]$. It follows that either $G/{[G,G]}\cong \mathbb{Z}_p^d$ or $G/{[G,G]}\cong \mathbb{Z}/p^f\mathbb{Z} \times \mathbb{Z}_p^{d-1}$ for some $f \geq 1$. Set $q=p^f$ in the latter and $q=0$ in the former case. Then $d$ and $q$ are two invariants associated to  $G$.  
%Indeed,  $q$ is the largest power of $p$ such that $r \in F^q[F, F]$. 
Demushkin groups have been classified completely by  Demushkin, Serre and Labute (see \cite{Demushkin1}, \cite{Demushkin2}, \cite{Serre2}, and  \cite{Labute}). We summarise the classification of Demushkin groups in the next theorem.

\begin{thm}
\label{demushkin - classification}
Let $G$ be a Demushkin  group and, let $F$, $d$, $q$, and $r$ be as above. 
\begin{itemize}
\item [(i)] If $q \neq 2$, there exists a basis $\{x_1, \cdots , x_d\}$ of $F$ such that  $$r=x_1^q[x_1,x_2]\cdots [x_{d-1}, x_d],$$  %where for $q=\infty$ we define $x_1^{\infty}=1$ 
and $d$ is necessarily even.
\item [(ii)] If $q=2$ and $d$ is odd, then there is a basis $\{x_1, \cdots , x_d\}$ of $F$ such that $$r=x_1^2x_2^{2^f}[x_2,x_3]\cdots [x_{d-1}, x_d]$$ for some $f= 2, 3, \ldots, \infty$ $( 2^{f}=0 \text{ when } f=\infty)$.
\item [(iii)] If $q=2$ and $d$ is even, then there is a basis $\{x_1, \cdots , x_d\}$ of $F$ such that either $r=x_1^{2^f+2}[x_1,x_2]\cdots [x_{d-1}, x_d]$ for some $f= 2, 3, \ldots, \infty$ or  $r=x_1^2[x_1,x_2]x_3^{2^f}[x_3, x_4]\cdots [x_{d-1}, x_d]$ for some  $f= 2, 3, \ldots, \infty$ and $d\geq 4$. 
\end{itemize}
\end{thm}

One of the main features of Demushkin groups is that every subgroup of finite index of a Demushkin group is also a Demushkin group and every subgroup of infinite index  is free pro-$p$ (cf. \cite{Serre1} or  \cite{Wingberg}). Since Demushkin groups are of cohomological dimension 2, using the  multiplicativity of the Euler-Poincar\'{e} characteristic one can show that $d(H)-2 = [G:H](d(G)-2)$ for every finite index subgroup $H$ of a Demushkin group $G$ (see \cite[Theorem 3.9.15]{Wingberg}).  Moreover, it is worth mentioning that a Demushkin group G is $p$-adic analytic if and only if $d(G)=2$.

%In this section we will discuss test elements in Demushkin groups. 

\begin{lem}\label{extraction-roots}
Let $G$ be a Demushkin group with $d(G)> 2$ and let $u$ and $v$ be elements in $G$. If $u^{\alpha} = v^{\alpha} $ for some $0 \neq \alpha \in \mathbb{Z}_p$, then $u=v$.
Moreover, if $u$ is a test element of $G$, then so is $u^{\alpha}$.
\end{lem}
\begin{proof}
Note that $\langle  u, v \rangle$ is a free pro-$p$ group of rank at most 2. Thus the results follow from unique extraction of roots in free pro-$p$ groups.
\end{proof}
%\begin{pro}
%\label{test-power-Demushkin}
%Let $G$ be a Demushkin group with $d(G)> 2$. If $u \in G$ is a test element of $G$, then so is $u^{\alpha}$ for every $0 \neq \alpha \in \mathbb{Z}_p$. 
%\end{pro}
%\begin{proof}
%Suppose that $u$ is a test element of $G$, and let $0 \neq \alpha \in \mathbb{Z}_p$.
%If $\varphi: G \to G$ is a homomorphism such that $\varphi(u^{\alpha}) = u^{\alpha}$, then by unique extraction of roots $\varphi(u) = u$. 
%Hence, $\varphi$ is an automorphism and $u^{\alpha}$ is a test element of $G$. 
%\end{proof}

Note that if $G$ is a torsion free pro-$p$ group then every proper retract of $G$ is of infinite index. To see this, suppose that $r: G \twoheadrightarrow H$ is a retraction, $H \lneq G$ and $[G:H]<\infty$. Let $g \in G\setminus H$. Then $g{r(g)}^{-1} \neq 1$ and $r(g{r(g)}^{-1}) = 1$. Thus $\textrm{ker}(r) \neq 1$ and so it is an infinite group. Since $[\textrm{ker}(r) : {H\cap \textrm{ker}(r)}] < \infty$, there exists some $1\neq h \in {H\cap \textrm{ker}(r)}$, which is  not possible.

\begin{lem}\label{retract-Demushkin}
Let $G$ be a Demushkin group with $d(G)=n$ and $H$ be a proper retract of $G$. Then $d(H) \leq \frac{n}{2}$.
\end{lem}
\begin{proof}
From the above observation we have that $H$ is of infinite index in $G$. Thus  $H$ is a free pro-$p$ group. Since $H$ is an epimorphic image of $G$, it follows from \cite{Sonn} that $d(H) \leq \frac{n}{2}.$
\end{proof}

\begin{lem}\label{free-Demushkin}
Let $G$ be a Demushkin group with $d(G)=n > 2$,  $\{ x_1, x_2, ..., x_n  \}$ a minimal generating set of $G$, and $k \leq n$  a positive integer. Then for each $2 \leq  l \leq k$, the pro-$p$ group $\langle   x_1^p,  x_2^p, ..., x_l^p, x_{l+1},..., x_k \rangle$ is a free pro-$p$ group of rank $k$. Moreover, if $k<n$, one can take $0 \leq  l \leq k$ .
\end{lem}
\begin{proof}
Let $H=  \langle   x_1^p,  x_2^p, ..., x_l^p, x_{l+1},..., x_k \rangle$.  Since $k-l < n$ and $x_i^p \in \Phi(G)$ for all $1\leq i \leq l$, it follows that $H \neq G$. Moreover,  $H$ is of infinite index in $G$ and therefore free. Indeed, if $ 1 \neq [G: H] < \infty$, then $d(H) = 2+ [G : H] (d(G) - 2)> n \geq d(H)$ , which is a contradiction. 
 Suppose that $G/{[G,G]}$ is not isomorphic to ${\mathbb{Z}/p\mathbb{Z}} \times \mathbb{Z}_p^{n-1}$. Then the cardinality of a minimal generating set of the  image of $ H $  in $G/{[G,G]}$ is $k$, which implies that $H$ is free of rank $k$. Now suppose that $G/{[G,G]} \cong {\mathbb{Z}/p\mathbb{Z}} \times \mathbb{Z}_p^{n-1}$. Then there is at most one $x_i$ such that $x_i^p$ is trivial modulo $[G,G]$. Without loss of generality we may assume that $i=1$. In the same way as above it can be seen that $K= \langle   x_1,  x_2^p, ..., x_l^p, x_{l+1},..., x_k \rangle$ is a free pro-$p$ group. Moreover, the cardinailty of a minimal generating set of the  image of $K$ in $G/{[G,G]}$ is $k$. Thus $K$ is a free pro-$p$ group of rank $k$. Let $M$ be a maximal subgroup of $K$ satisfying $\{  x_2^p, ..., x_l^p, x_{l+1},..., x_k \} \subseteq M$ but $x_1 \notin M$. Then using as a Schreier transversal for $M$ the set $\{1, x_1, x_1^2, ..., x_1^{p-1} \}$, one can easily see that $\{x_1^p,  x_2^p, ..., x_l^p, x_{l+1},..., x_k \}$ is a part of a basis of $M$. Hence, $H$ is a free pro-$p$ group of rank $k$.

\end{proof}

\begin{thm}\label{Demushkin-test}
Let $G$ be a Demushkin group with $d(G)=n > 2$ and $\{ x_1, x_2, ..., x_n  \}$ be a minimal generating set of $G$. Fix a positive integer $k$ such that $\left \lfloor{\frac{n}{2}}\right \rfloor + 1 \leq k  \leq n $. For each $1\leq i \leq k$, let $\alpha_i \in \mathbb{Z}_p \setminus \{0\}$ and suppose that at least $ k- \left \lfloor{\frac{n}{2}}\right \rfloor  +1 $ $\alpha_i$'s belong to $p\mathbb{Z}_p \setminus \{0\}$. Moreover, if $p=2$ and $n \leq 4$, suppose that $k<n$. If $w(y_1, y_2, ..., y_k)$ is a test element of the free pro-$p$ group $F(y_1, y_2, ..., y_k)$, then $t = w(x_1^{\alpha_1}, x_2^{\alpha_2}, ...,  x_k^{\alpha_k})$ is a test element of $G$.
\end{thm}
\begin{proof}
Without loss of generality we may assume that $\alpha_1, \alpha_2, ..., \alpha_{k -  \left \lfloor{\frac{n}{2}}\right \rfloor +1} \in p\mathbb{Z}_p$. Suppose that  $t = w(x_1^{\alpha_1}, x_2^{\alpha_2}, ...,  x_k^{\alpha_k})$ is not a test element of $G$. Then by Theorem \ref{test-retract}, $t$ belongs to a proper retract $R$ of $G$. By Lemma \ref{retract-Demushkin}, we have $s = d(R) \leq \left \lfloor{\frac{n}{2}}\right \rfloor $. Let $\{ u_1, u_2, ..., u_s \}$ be a minimal generating set of $R$ and consider the group 
$$H =\langle x_1^p,  x_2^p, ..., x_{k - \left \lfloor{\frac{n}{2}}\right \rfloor  +1}^p, x_{k-\left \lfloor{\frac{n}{2}}\right \rfloor  +2},..., x_k, u_1, ..., u_s \rangle.$$  
Note that $H \neq G$ since $x_i^p \in \Phi(G)$ for all $1 \leq i \leq  k - \left \lfloor{\frac{n}{2}}\right \rfloor + 1$ and $s + [k - (k - \left \lfloor{\frac{n}{2}}\right \rfloor  +2) +1 ]< n$. We claim that $H$ is not an open subgroup of $G$. Indeed, if $H$ is a proper open subgroup of $G$, then 
\begin{displaymath}
d(H) = 2 + [G:H](d(G) - 2) \geq 2 + [G:H](n - 2)  > n + \left \lfloor{\frac{n}{2}}\right \rfloor  \geq k+s,
\end{displaymath}
unless $p=2$ and $n \leq 4$, in which case by assumption $k < n$ and thus $d(H) \geq n + \left \lfloor{\frac{n}{2}}\right \rfloor  > k+s$. Since $k+s \geq d(H)$, we have a contradiction. Hence, $H$ has infinite index in $G$ and thus it is a free pro-$p$ group.  Moreover, by Lemma \ref{retract-intersection} $(b)$, $R$ is a retract of $H$. Let $A =\langle  x_1^p,  x_2^p, ..., x_{k - \left \lfloor{\frac{n}{2}}\right \rfloor +1}^p, x_{k- \left \lfloor{\frac{n}{2}}\right \rfloor +2},..., x_k \rangle$.  We claim that $R \cap A \neq A$. Suppose to the contrary that $A \subseteq R$. Then $ x_i^p \in R$ for all $1 \leq i \leq  k - \left \lfloor{\frac{n}{2}}\right \rfloor + 1$. Let $r: G \to R$ be a retraction. Then $ x_i^p = r( x_i^p) =  {r(x_i)}^p$ , so by Lemma \ref{extraction-roots}, $x_i = r(x_i) \in R$ for all $1 \leq i \leq k$.   
Since $\{x_1, \ldots, x_k\}$ is a part of a minimal generating set of $G$, it is also a part of a minimal generating set of $R$, which is impossible since $s \leq \left \lfloor{\frac{n}{2}}\right \rfloor<k$. Hence, by Proposition \ref{freefactors-intersection} and Corollary \ref{test-free},  $A \cap R$ is a proper retract of $A$ and  $t = w(x_1^{\alpha_1}, x_2^{\alpha_2}, ...,  x_k^{\alpha_k}) \in A\cap R$ is not a test element of $A$.  However, $$t = w( {(x_1^p)}^{\beta_1}, {(x_2^p)}^{\beta_2}, ...,  {(x_{k - \left \lfloor{\frac{n}{2}}\right \rfloor +1}^p)}^{\beta_{k - \left \lfloor{\frac{n}{2}}\right \rfloor +1}}, x_{{k - \left \lfloor{\frac{n}{2}}\right \rfloor +2}}^{\alpha_{{k - \left \lfloor{\frac{n}{2}}\right \rfloor +2}}}, ..., x_k^{\alpha_k})$$ for some $\beta_j \in \mathbb{Z}_p \setminus \{0 \}, 1 \leq j \leq k -  \left \lfloor{\frac{n}{2}}\right \rfloor +1$. By  Lemma \ref{free-Demushkin}, $A$ is a free pro-$p$ group of rank $k$. Therefore, by Corollary \ref{test_powers},  $t$ is a test element of $A = F(x_1^p,  x_2^p, ..., x_{k - \left \lfloor{\frac{n}{2}}\right \rfloor +1}^p, x_{k- \left \lfloor{\frac{n}{2}}\right \rfloor +2},..., x_k)$, which is a contradiction. Hence, $t$ is a test element of $G$. 
\end{proof}

%\begin{cor}
%Let $G$ be a Demushkin group with $d(G)=n > 3$, let $\{ x_1, x_2, ..., x_n  \}$ be a minimal generating set of $G$ and if $p=2$, suppose that $n>4$. Fix a positive integer $k$ such that $\left \lfloor{\frac{n}{2}}\right \rfloor + 1 \leq k  \leq n $. For each $1\leq i \leq k$ let $\alpha_i \in p\mathbb{Z}_p \setminus \{0\}$ and suppose that at least $ k- \left \lfloor{\frac{n}{2}}\right \rfloor  +1 $ $\alpha_i$'s belong to $p^2\mathbb{Z}_p \setminus \{0\}$. Then $t = x_1^{\alpha_1}x_2^{\alpha_2}\cdots x_k^{\alpha_k}$ is a test element of $G$.
%\end{cor}

\begin{thm}\label{Demushkin-test-2}
Let $G$ be a Demushkin  pro-$2$ group with $3 \leq d(G) = n \leq 4$, and let $\{ x_1, x_2, ..., x_n  \}$ be a minimal generating set of $G$. For each $1\leq i \leq n$ let $\alpha_i \in 2\mathbb{Z}_2 \setminus \{0\}$. If $w(y_1, y_2, ..., y_n)$ is a test element of the free pro-$2$ group $F(y_1, y_2, ..., y_n)$, then $t = w(x_1^{\alpha_1}, x_2^{\alpha_2}, ...,  x_n^{\alpha_n})$ is a test element of $G$.
\end{thm}
\begin{proof}
 Suppose that  $t = w(x_1^{\alpha_1}, x_2^{\alpha_2}, ...,  x_n^{\alpha_n})$ is not a test element of $G$. Then by Theorem \ref{test-retract}, $t$ belongs to a proper retract $R$ of $G$. By Lemma \ref{retract-Demushkin}, $s = d(R) \leq \left \lfloor{\frac{n}{2}}\right \rfloor $. Let $\{ u_1, u_2, ..., u_s \}$ be a minimal generating set of $R$ and consider the group $H = \langle x_1^p,  x_2^p, ..., x_n^p, u_1, ..., u_s \rangle$.  Note that $H \neq G$, since $x_i^p \in \Phi(G)$ for all $1 \leq i \leq  n$ and $s  < n$. If $H$ is a proper open subgroup of $G$ with $[G:H] > 2$, then
\begin{displaymath}
d(H) = 2 + [G:H](d(G) - 2)  > 2 + 2(n-2) \geq n + \left \lfloor{\frac{n}{2}}\right \rfloor  \geq n+s \geq d(H),
\end{displaymath}
which is a contradiction. Now assume that $[G:H] =2$, i.e., $H$ is a maximal subgroup of $G$. Note that the minimal generating set of the image of $H$  in  $G/{\Phi(G)}$ is of cardinality at most $\left \lfloor{\frac{n}{2}}\right \rfloor$. However,  the minimal generating set of the image in $G/{\Phi(G)}$ of any maximal subgroup  of $G$ is of cardinality $n-1$. Thus $H$ can not be a maximal subgroup of $G$. Hence, $H$ has infinite index in $G$, and thus it is a free pro-$2$ group.  Moreover, by Lemma \ref{retract-intersection}, $R$ is a retract of $H$. Let $A = \langle  x_1^p,  x_2^p, ..., x_n^p \rangle$. Then arguing as in the proof of Theorem \ref{Demushkin-test}, on one side we get that $R \cap A$ is a proper retract of $A$ that contains $t$ and on the other side we get that $t$ is a test element of $A$, which is a contradiction by Theorem \ref{test-retract}.   Hence, $t$ is a test element of $G$. 
\end{proof}

\section{Test elements of discrete groups}

Let $G$ be a discrete group and $p$ be a prime. We denote by $\jmath_{p}:G \to \widehat{G}_{p}$ the pro-$p$ completion of $G$. 
If $G$ is residually finite-$p$, then $\jmath_p$ is an embedding and we identify $\jmath_{p}(G)$ with $G$. 
\begin{pro}
\label{pro-p - discrete}
Let $p$ be a prime, and let $G$ be a finitely generated residually-$p$ Turner group. 
If $g \in G$ is a test element of $\widehat{G}_p$, then $g$ is a test element of $G$. 
\end{pro}
\begin{proof}
Let $g\in G$ be a test element of $\widehat{G}_p$, and assume to the contrary that $g$ is not a test element of $G$.
Since $G$ is a Turner group, $g$ is contained in a proper retract $H$ of $G$ (see the Introduction for the definition of a Turner group). Let $r:G \to H$ be a retraction, and consider the following diagram of pro-$p$ groups and continuous homomorphisms:
\[\xymatrix{ \widehat{H}_p \ar[r]^{\imath_p} \ar[dr]_{id} & \widehat{G}_p \ar[d]^{r_p} \\
 & \widehat{H}_p \ar[r]^-{\imath_p} & \overline{H} \subseteq \widehat{G}_p  }\] 
Here $\imath:H \to G$ is the inclusion homomorphism. Clearly, $\imath_p$ is an injection and $\imath_p(\widehat{H}_p)=\overline{H}$ (see \cite[Theorem~3.2.4]{Ribes1} ).
If $\imath_p(x) \in \overline{H}$, then $(\imath_p \circ r_p)(\imath_p(x))=\imath_p(x)$. Hence, $\imath_p \circ r_p:\widehat{G}_p \to \overline{H}$ is a retraction.
Since $H$ is a proper retract of $G$, $r$ is not injective. Consequently, $r_p$ is not injective and 
$\overline{H}$ is a proper retract of $\widehat{G}_p$. As $g \in \overline{H}$,
it follows from Theorem~\ref{test-retract} that $g$  is not a test element of $\widehat{G}_p$, a contradiction. 
\end{proof}

\begin{cor}
\label{pro-p - free}
Let $D$ be a free discrete group of finite rank and $p$ be a prime. If $w \in D$ is a test element of $\widehat{D}_p$ (a free pro-$p$ group of the same rank as $D$), then $w$ is a test element of $D$.
\end{cor}

We can apply Corollary~\ref{pro-p - free} to the examples of test elements  of free pro-$p$ groups (represented by words of finite length) described in Section~\ref{test-free pro-p}, and thus obtain many new examples of test elements of free discrete groups; 
in the following proposition we give only a sample. (To the best of our knowledge the test elements described in parts $(b)$ and $(c)$ are new.)

\begin{pro}
\label{examples - discrete}
Let $D$ be a free discrete group with basis $\{x_1, \ldots, x_n\}$, $n\geq 2$. Then the following holds:
\begin{enumerate}[(a)]
\item Every higher commutator $t_1=[x_1, ... , x_n]$ is a test element of $D$.
\item Let $n=m+2k$ for some $m, k \geq 0$. Then 
\[t_2=x_1^{s_1} \ldots x_m^{s_m}[x_{m+1}^{r_1}, x_{m+2}^{r_2}] \ldots [x_{m+2k-1}^{r_{2k-1}}, x_{m+2k}^{r_{2k}}]\]
is a test element of $D$ if and only if  $\gcd(s_1, \ldots, s_m) \neq 1$, $ s_i \neq 0$ for all $1 \leq i \leq m$, and $r_j \neq 0$ for all $1 \leq j \leq 2k$ . Moreover, if for some prime $p$, $s_i \in p\mathbb{Z} \setminus p^2\mathbb{Z}$ for all $1 \leq i \leq m$ and $r_i=1$ for $1 \leq i \leq 2k$, then
for every $w \in (D^p[D,D])^p[D^p[D,D], D^p[D,D]]$, $t_2w$ is a test element of $D$.
\item If $n$ is even, then
\[t_3=x_1^{s_1} \ldots x_{n}^{s_n}[x_1,x_2] \ldots [x_{n-1}, x_{n}]\]
is a test element of $D$ if and only if $\gcd(s_1, \ldots, s_n) \neq 1$.
Moreover, if $t_3$ is a test element of $D$, then
for every $w \in (D^p[D,D])^p[D^p[D,D], D^p[D,D]]$, $t_3w$ is also a test element of $D$.
\end{enumerate}
\end{pro}
\begin{proof}
$(a)$ By Proposition~\ref{higher commutator}, $t_1$ is a test element of $\widehat{D}_p$ for every prime $p$. Hence, by Corollary~\ref{pro-p - free}, $t_1$ is a test element of $D$.
\\
$(b)$ If  $\gcd(s_1, \ldots, s_m)=1$ (so $m> 0$), then there exist integers $l_1, \ldots, l_m$ such that
$s_1l_1 + \ldots + s_ml_m=1$. Then the homomorphism $r:D \to \langle t \rangle$ defined by $r(x_i)=t_2^{l_i}$ for $1 \leq i \leq m$ and $r(x_i)=1$ for $m < i \leq n$ is clearly a retraction. Hence, $t_2$ is not a test element of $D$.
Furthermore, if  $s_i = 0$ for some $1\leq i \leq m$, then $t_2 \in\langle x_1, \ldots, x_{i-1}, x_{i+1}, \ldots, x_n \rangle$, which is a proper retract of $D$. Similarly, if $r_j=0$ for some $1 \leq j \leq 2k$, then
$t_2 \in\langle x_1, \ldots, x_{m+j-1}, x_{m+j+1}, \ldots, x_n \rangle$, which is also a proper retract of $D$.

Now assume that $\gcd(s_1, \ldots, s_m) \neq 1, s_i \neq 0$ for all $1 \leq i \leq m$, and $r_j \neq 0$ for all $1 \leq j \leq 2k$. Let $p$ be any prime that divides $\gcd(s_1, \ldots, s_m)$.
By Proposition~\ref{almost_primitive - examples} $(d)$, $t_2$ is a test element of $\widehat{D}_p$. Hence, by Corollary~\ref{pro-p - free}, $t_2$ is a test element of $D$.
Moreover, if $s_i \in p\mathbb{Z} \setminus p^2\mathbb{Z}$ for $1 \leq i \leq m$ and $r_i=1$ for $1 \leq i \leq 2k$, then by Proposition~\ref{almost_primitive - examples} $(c)$, $t_2$ is an almost primitive element of $\widehat{D}_p$.
Observe that 
\[(D^p[D,D])^p[D^p[D,D], D^p[D,D]] \subseteq \Phi^2(\widehat{D}_p) \subseteq \bigcap_{M}\Phi(M)\]
where the intersection is over all maximal subgroups of $\widehat{D}_p$. Hence,
it follows from Proposition~\ref{almost primitive - modify} $(a)$ and Corollary~\ref{pro-p - free} that $t_2w$ is a test element of $D$ for every $w \in (D^p[D,D])^p[D^p[D,D], D^p[D,D]]$.
\\
$(c)$ The proof is the same as for part $(b)$. Here we use Remark~\ref{almost primitive remark}. 
\end{proof}

\begin{pro}
Let $D$ be a free discrete group with basis $\{x_1, \ldots, x_n\}$, and let $s_i \in \mathbb{Z} \setminus \{0\}$ for $1\leq i \leq n$. If $w \in D$ is a test element of $\widehat{D}_p$ for some prime $p$, then $w(x_1^{s_1}, \ldots, x_n^{s_n})$ is a test element of $D$.
\end{pro}
\begin{proof}
Apply Proposition~\ref{test_powers} and Corollary~\ref{pro-p - free}. 
\end{proof}
\smallskip

Recall that the \emph{profinite topology} on a group $G$ has the following basis: 
\[\{gN \mid g \in G, N \unlhd G \text{ and } |G:N| < \infty\}.\] 

\begin{thm}
\label{dense - discrete}
Let $D$ be a free discrete group with finite basis $\{x_1, \ldots, x_n\}$. Then the set of test elements of $D$ is a dense subset of $D$ in the profinite topology. 
\end{thm}
\begin{proof}
Let $X$ be the closure in the profinite topology on $D$ of the set of test elements of $D$.
We need to show that $X=D$. Let $w \in D $ and $N \unlhd D$ with $|D:N| = m < \infty$.
Choose a prime $p$ such that $p \nmid m$.
For each $1 \leq i \leq n$, let $k_i$ be the sum of the exponents of all occurrences of $x_i$ in $w$. Then we can find natural numbers 
$r_i \geq 0$ satisfying $p \mid k_i + r_im$ for all $1\leq i \leq n$; set $w'=wx_1^{r_1m} \ldots x_n^{r_nm}$.
Then $wN=w'N$ and $w' \in \Phi(\widehat{D}_p)$ ($\sigma_{x_i}(w') \in p\mathbb{Z}_p$ for all $1 \leq i \leq n$).
By Lemma~\ref{a lot of test elements} $(b)$, there exists $\{x_{i_1}, \ldots, x_{i_k}\} \subseteq \{x_1, \ldots, x_n\}$ such that
for every test element $u$ of $H=\overline{\langle  x_{i_1}, \ldots, x_{i_k} \rangle} \leq \widehat{D}_p$ contained in $\Phi(H)$, $w'u$ is a test element of $\widehat{D}_p$.
By Proposition~\ref{almost_primitive - examples} $(d)$, $u=x_{i_1}^{pm} \ldots  x_{i_k}^{pm} \in N$ is a test element of $H$ and clearly $u \in \Phi(H)$. Hence, $t=w'u \in wN$ is a test element of $\widehat{D}_p$.
It follows from Corollary~\ref{pro-p - free} that $t$ is a test element of $D$; thus $X=D$.
\end{proof}

One might wonder whether every test element of a free discrete group $D$ is a test element of $\widehat{D}_p$ for some prime $p$. This is certainly the case for free groups of rank one, and we prove in the following proposition that it is also true for free groups of rank two.
However, we don't know if it is true for all free discrete groups of finite rank.  

\begin{pro}
Let $D$ be a free discrete group of rank two, and let $w \in D$. Then $w$ is a test element of $D$ if and only if there exists a  prime $p$ such that $w$ is a test element of $\widehat{D}_p$. 
\end{pro}
\begin{proof}
Suppose that for some prime $p$, $w$ is a test element of $\widehat{D}_p$.
By Corollary~\ref{pro-p - free}, $w$ is a test element of $D$.
Now suppose that $w$ is a test element of $D$. Then $w=u^n$ for some $u \in D$ that is not a proper power.
Let $\{x_1, x_2\}$ be a basis for $D$. Denote by $k_1$  (resp. $k_2$) the sum of the exponents of all occurrences of $x_1$ (resp. $x_2$) in $u$.
It is easy to see that $d=\gcd(k_1, k_2) \neq 1$. Indeed, otherwise $m_1k_1 + m_2k_2=1$ for some $m_1, m_2 \in \mathbb{Z}$, and the homomorphism
$r:F \to \langle u \rangle$  defined by $r(x_1)=u^{m_1}$ and $r(x_2)=u^{m_2}$ is a retraction.
Let $p$ be a prime divisor of $d$. Then $\sigma_{x_1}(u), \sigma_{x_2}(u) \in p \mathbb{Z}_p$. It follows from Theorem 2 in \cite{Baumslag} that $u$ is not a $p$-th power 
in $\widehat{D}_p$. By Proposition~\ref{rank2}, $u$ is a test element of $\widehat{D}_p$. Moreover, it follows from Proposition~\ref{test-power} that $w$ is also a test element of $\widehat{D}_p$.
\end{proof}

We now turn to test elements in surface groups. For $n \geq 1$, let 
$$G(n)=\langle x_1, \ldots, x_{2n} \mid [x_1, x_2] \ldots [x_{2n-1}, x_{2n}] \rangle,$$
i.e., $G(n)$ is an orientable surface group of genus $n$. O'Neill and Turner \cite{O'Neill} proved that every orientable surface group is a Turner group. Moreover, Baumslag \cite{Baumslag1} showed that orientable surface groups are residually finite-$p$ for every prime $p$. 
Hence, we have the following corollary to Proposition~\ref{pro-p - discrete}.

\begin{cor}
\label{pro-p - surface}
Let $n \geq 1$ be a natural number.
If $w \in G(n)$ is a test element of $\widehat{G(n)}_p$ for some prime $p$, then it is a test element of $G(n)$.
\end{cor}

It follows from \cite[Lemma 2.1]{Lu} and Theorem~\ref{demushkin - classification} $(i)$ that $\widehat{G(n)}_p$ is a Demushkin group for every $n \geq 1$ and every prime $p$.
Therefore, the following two propositions are immediate consequences of Corollary~\ref{pro-p - surface}, Theorem~\ref{Demushkin-test} and Theorem~\ref{Demushkin-test-2}.

\begin{pro}
\label{surface - test}
Let $n \geq 2$ be a natural number, $p$ a prime, and $\{a_1, \ldots, a_{2n}\}$ any generating set of $G(n)$.
Fix a positive integer $k$ such that $n + 1 \leq k  \leq 2n $. Let $s_1, \ldots, s_k$ be non-zero integers, and suppose that at least $ k -  n  +1 $ $s_i$'s are divisible by $p$. Moreover, if $p=2$ and $n = 4$, suppose that $k < 2n$. 
Let $D$ be a free discrete group on $\{y_1, \ldots, y_k\}$.
If $w(y_1, y_2, ..., y_k) \in D$ is a test element of $\widehat{D}_p$, then $t = w(a_1^{s_1}, a_2^{s_2}, ...,  a_k^{s_k})$ is a test element of $G(n)$.
\end{pro}

\begin{pro}
\label{surface - test - 2}
Let $\{a_1, a_2, a_3, a_4\}$ be any generating set of $G(2)$, and let $s_1, s_2, s_3,  s_4$ be even non-zero integers. Let $D$ be a free discrete group on $\{y_1, y_2, y_3, y_4\}$.
If $w(y_1, y_2, y_3, y_4) \in D$ is a test element of the free pro-$2$ group $\widehat{D}_2$, then $t = w(a_1^{s_1}, a_2^{s_2},a_3^{s_3},  a_4^{s_4})$ is a test element of $G(2)$.
\end{pro}

Non-orientable surface groups of genus $n\geq 3$ are Turner groups (see \cite{O'Neill}). 
Moreover, non-orientable surface groups of genus $n\geq 4$ are residually free, and therefore residually finite-$p$ for every prime $p$ (see \cite{Benjamin-Baumslag} and \cite{Baumslag1}).  
By \cite[Lemma~8.9]{Minasyan}, the non-orientable surface group of genus $3$ is also residually finite-$p$ for all $p$.
Hence, non-orientable surface groups of genus $n\geq 3$ satisfy the hypothesis of Proposition~\ref{pro-p - discrete}.
Now assume that $p \geq 3$, and let $G=\langle x_1, \ldots, x_n \mid x_1^2 \ldots x_n^2 \rangle$ be a non-orientable surface group of genus $n \geq 3$. 
By \cite[Lemma 2.1]{Lu}, $\widehat{G}_p$ has a presentation $\langle x_1, \ldots, x_n \mid x_1^2 \ldots x_n^2 \rangle$ as a pro-$p$ group. 
Observe that $t= x_1^2 \ldots x_n^2$ is a primitive element in the free pro-$p$ group $F(x_1, \ldots, x_n)$.
Hence, $\widehat{G}_p$ is a free pro-$p$ group of rank $n-1$, and we have the following proposition. 

\begin{pro}
\label{non-orientable surface groups}
Let $G=\langle x_1, \ldots, x_n \mid x_1^2 \ldots x_n^2 \rangle$ be a non-orientable surface group of genus $n \geq 3$, and
let $D$ be a free discrete group on $\{y_1, \ldots, y_{n-1}\}$. 
If $w(y_1, y_2, ..., y_{n-1}) \in D$ is a test element of $\widehat{D}_p$ for some $p \geq 3$, then $t = w(x_{i_1}, x_{i_2}, ...,  x_{i_{n-1}})$ is a test element of $G$ for any choice of distinct 
elements $x_{i_j} \in \{x_1, \ldots, x_n\}, 1\leq j\leq n-1$.
\end{pro}

\ackn 
This work was carried out while the second author was holding a CNPq Postdoctoral Fellowship at the Federal University of Rio de Janeiro. He would like to thank CNPq for the financial support and the Department of Mathematics at UFRJ  for its warm hospitality and the excellent research environment.

\bibliographystyle{plain}

\end{document}